\documentclass[11pt]{article}

\usepackage{amsfonts}
\usepackage{amstext}
\usepackage{amsthm}
\usepackage{amsmath}
\usepackage{amssymb}
\usepackage{latexsym}
\usepackage{graphicx}
\usepackage[latin1]{inputenc}
\usepackage[colorlinks=true,linkcolor=blue,citecolor=red]{hyperref}

\newcommand{\R}{\Bbb{R}}

\newcommand{\vs}{\vspace*}
\newcommand{\n}{{\noindent}}

\DeclareMathOperator{\Hessian}{Hess}

\usepackage{graphicx}

\newtheorem{lem}{Lemma}[section]

\newtheorem{remark}{Remark}[section]
\newtheorem{defi}{Definition}[section]
\newtheorem{exam}{Example}[section]
\newtheorem{corol}{Corollary}[section]
\newtheorem{theo}{Theorem}[section]
\newtheorem{prop}{Proposition}[section]

\begin{document}

\title{Extremal solutions for a broad spectrum of nonlinear elliptic systems
}

\author{\textbf{Felipe Costa \footnote{\textit{E-mail addresses}:
felipemat@ufmg.br (F. Costa)}}\\ {\small\it Departamento de Matem\'{a}tica,
Universidade Federal de Minas Gerais,}\\ {\small\it Caixa Postal
702, 30123-970, Belo Horizonte, MG, Brazil}\\
\textbf{Gil F. de Souza \footnote{\textit{E-mail addresses}:
gilsouza@ufop.edu.br (G. F. de Souza)}}\\ {\small\it Departamento de Matem\'{a}tica,
Universidade Federal de Ouro Preto,}\\ {\small\it ICEB, 35400-000, Ouro Preto, MG, Brazil}\\
\textbf{Marcos Montenegro \footnote{\textit{E-mail addresses}:
montene@mat.ufmg.br (M. Montenegro)}}\\ {\small\it Departamento de Matem\'{a}tica,
Universidade Federal de Minas Gerais,}\\ {\small\it Caixa Postal
702, 30123-970, Belo Horizonte, MG, Brazil}}

\date{}{

\maketitle

\markboth{abstract}{abstract}
\addcontentsline{toc}{section}{Abstract}

\hrule \vspace{0,2cm}

\n {\bf Abstract}

In this paper we study positive solutions for the Dirichlet problem
\[
\left\{
\begin{array}{rlllr}
-{\cal L} u &=& \Lambda F(x,u) & {\rm in} & \Omega, \\

u&=&0 & {\rm on} & \partial \Omega,
\end{array}\right.
\]
where $\Omega$ is a smooth bounded domain in $\mathbb{R}^n$, $n\geq2$, $u=(u_1,\ldots,u_m): \overline{\Omega} \to \mathbb{R}^m$, $m \geq 1$, ${\cal L} u =({\cal L}_1 u_1, \ldots, {\cal L}_m u_m)$, where each ${\cal L}_i$ denotes a uniformly elliptic linear operator of second order in nondivergence form in $\Omega$, $\Lambda = (\lambda_1,\ldots, \lambda_m) \in \R^m$, $F=(f_1,\ldots,f_m): \Omega \times \R^m \to \mathbb{R}^m$ and $\Lambda F(x,u)=(\lambda_1f_1(x,u),\ldots,\lambda_mf_m(x,u))$.

For a general class of maps $F$ we prove that there exists a hypersurface $\Lambda^*$ in $\R^m_+ := (0,\infty)^m$ such that tuples $\Lambda \in \R^m_+$ below $\Lambda^*$ correspond to minimal positive strong solutions of the above system. Stability of these solutions is also discussed. Already for tuples above $\Lambda^*$, there is no nonnegative strong solution. The shape of the hypersurface $\Lambda^*$ depends on growth on $u$ of the vector nonlinearity $F$ in a precise way to be addressed. When $\Lambda \in \Lambda^*$ and the coefficients of each operator ${\cal L}_i$ are slightly smooth, the problem admits a unique minimal nonnegative weak solution, called extremal solution. Furthermore, when $F$ is a potential field and each ${\cal L}_i$ is the Laplacian, we investigate the strong regularity of this solution for any $m \geq 1$ in dimensions $n = 2$ and $n = 3$ for convex domains and $2 \leq n \leq 9$ for balls.

\vspace{0.5cm}
\hrule\vspace{0.2cm}





\setcounter{page}{1}

\vs{0.20cm}

\section{Introduction and main statements}

The present work concerns mainly with existence, regularity and stability of positive strong solution for the elliptic system to $m$-parameters with zero boundary data:

\[
\left\{
\begin{array}{rlllrl}
-{\cal L}_i u_i &=& \lambda_i f_i(x,u_1, \ldots, u_m) & {\rm in} & \Omega, & i=1, \ldots, m, \\

u_i&=&0 & {\rm on} & \partial \Omega,& i=1, \ldots, m,
\end{array}\right.
\]
where $\Omega$ is a bounded open subset of $\mathbb{R}^n$ with $C^{1,1}$ boundary, $n\geq2$, $f_i: \Omega \times \R^m \rightarrow \R$ is a Carathéodory function, $m \geq 1$, $\lambda_i$ is a positive number and ${\cal L}_i$ denotes a linear differential operator of second order in $\Omega$ of the form

\[
{\cal L}_i = a^i_{kl}(x) \partial_{kl} + b^i_j(x) \partial_j + c^i(x)
\]
for every $i = 1, \ldots, m$. We assume that each ${\cal L}_i$ is uniformly elliptic, that is, there exist positive constants $C_0$ and $c_0$ such that

\[
c_0 |\xi|^2 \leq a^i_{kl}(x) \xi_k \xi_l \leq C_0 |\xi|^2,\ \ \forall \xi \in \R^n,\ x \in \Omega,\ i = 1, \ldots, m \, .
\]
We also consider coefficients $a^i_{kl} \in C(\overline{\Omega})$, $b^i_j, c^i \in L^\infty(\Omega)$ and a constant $b > 0$ such that

\[
|b^i_j(x)|, |c^i(x)| \leq b,\ \ \forall x \in \Omega,\ i = 1, \ldots, m\, .
\]
Moreover, assume that each ${\cal L}_i$ satisfies the strong maximum principle in $\Omega$, or equivalently, its principal eigenvalue $\mu_1(-{\cal L}_i, \Omega)$ is positive.

Throughout paper, the above problem will be represented shortly as

\begin{gather}\tag{1.1} \label{1}
\left\{
\begin{array}{rlllr}
-{\cal L} u &=& \Lambda F(x,u) & {\rm in} & \Omega, \\
u&=&0 & {\rm on} & \partial \Omega
\end{array}\right.
\end{gather}
with the natural notations $u=(u_1,\ldots,u_m)$, ${\cal L} u =({\cal L}_1 u_1, \ldots, {\cal L}_m u_m)$, $\Lambda = (\lambda_1,\ldots, \lambda_m) \in \R^m_+ := (0,\infty)^m$, $F=(f_1,\ldots,f_m)$ and $\Lambda F(x,u)=(\lambda_1f_1(x,u),\ldots,\lambda_mf_m(x,u))$.

For $m = 1$, the classical Dirichlet problem

\begin{gather}\tag{1.2} \label{2}
\left\{
\begin{array}{rlllr}
-{\cal L} u &=& \lambda f(u) & {\rm in} & \Omega, \\
u&=&0 & {\rm on} & \partial \Omega
\end{array}\right.
\end{gather}
has been well studied since the 1960s for $C^1$ functions $f: \R \rightarrow \R$ satisfying:

\begin{itemize}
\item[(a)] $f(0) > 0$ ($f$ is positive at the origin);
\item[(b)] $f(s) \leq f(t)$ for every $0 \leq s \leq t$ ($f$ is nondecreasing);
\item[(c)] $\lim_{t \rightarrow \infty} \frac{f(t)}{t} = \infty$ ($f$ is superlinear at infinity).
\end{itemize}
In fact, it all started with the case treated by Gelfand in his celebrated work \cite{Gelfand1963} where ${\cal L}$ is the Laplace operator $\Delta$, $f(u) = e^u$ and $\Omega$ is a ball $B$. For more general operators, under the conditions (a)-(c), Keller and Cohen \cite{KC}, Keller and Keener \cite{KK} and Crandall and Rabinowitz \cite{Ra} established the existence of a positive parameter $\lambda^*$ such that for any $0 < \lambda < \lambda^*$, the problem \eqref{2} admits a minimal positive strong solution $u_\lambda$, which is non-decreasing and continuous on $\lambda \in (0, \lambda^*)$ and, moreover, it is differentiable on $\lambda$ and asymptotically stable provided that $f$ is convex, see \cite{Du}. Asymptotic stability here means that the principal eigenvalue $\mu_1(-\Delta - \lambda f^\prime(u_\lambda))$ corresponding to the linearized problem with zero boundary condition is positive. Already for $\lambda > \lambda^*$, they showed that there is no nonnegative strong solution. When ${\cal L} = \Delta$, by taking the pointwise limit of $(u_\lambda)$ as $\lambda \uparrow \lambda^*$, Brezis, Cazenave, Martel and Ramiandrisoa proved (see Lemma 5 of \cite{BCMR}) that a unique minimal nonnegative weak solution $u^* \in L^1(\Omega)$ exists for $\lambda = \lambda^*$, so called extremal solution of \eqref{2}, and also by assuming convexity of $f$ (or that $\Omega$ is a ball, see \cite{Du}), no nonnegative weak solution exist for $\lambda > \lambda^*$.

By adapting some ideas of \cite{BCMR}, we easily deduce that an extremal solution also exists for operators in nondivergence form with slightly smooth coefficients. Indeed, consider a uniformly elliptic operator ${\cal L} = a_{kl}(x) \partial_{kl} + b_j(x) \partial_j + c(x)$, with $a_{kl} \in C^2(\overline{\Omega})$ and $b_j \in C^1(\overline{\Omega})$, satisfying maximum principle. Denote by ${\cal L}^*$ the adjoint operator of ${\cal L}$. Note that $0 < \mu_1 := \mu_1(-{\cal L}, \Omega) = \mu_1(-{\cal L}^*, \Omega)$. Let $\varphi^* \in W^{2,n}(\Omega) \cap W_0^{1,n}(\Omega)$ be a positive eigenfunction of $-{\cal L}^*$ associated to $\mu_1$. Thanks to (a) and (c), there exists a constant $C > 0$ such that $f(t) \geq \frac{2 \mu_1}{\lambda^*} t - C$ for all $t \geq 0$. Using this inequality, we get

\[
\lambda \int_\Omega f(u_\lambda) \varphi^* dx = \int_\Omega u_\lambda (-{\cal L}^* \varphi^*) dx = \mu_1 \int_\Omega u_\lambda \varphi^* dx \leq \frac{1}{2} \lambda^* \int_\Omega f(u_\lambda) \varphi^* dx + \frac{1}{2} \lambda^* C \int_\Omega \varphi^* dx\, .
\]
Letting $\lambda \uparrow \lambda^*$ and using (b), we obtain that $\lim_{\lambda \rightarrow \lambda^*} \int_\Omega f(u_\lambda) \varphi^* dx$ exists and is finite.

We now take the positive strong solution $\zeta^* \in W^{2,n}(\Omega) \cap W_0^{1,n}(\Omega)$ of the problem

\[
\left\{
\begin{array}{rlllr}
-{\cal L}^* u &=& 1 & {\rm in} & \Omega, \\

u&=&0 & {\rm on} & \partial \Omega
\end{array}\right.
\]
and notice by Hopf's Lemma that $\zeta^* \leq C \varphi^*$ in $\Omega$ for some constant $C > 0$. Then, we obtain

\[
\int_\Omega u_\lambda dx = \lambda \int_\Omega f(u_\lambda) \zeta^* dx \leq \lambda^* C \int_\Omega f(u_\lambda) \varphi^* dx\, ,
\]
so that $(u_\lambda)$ is uniformly bounded in $L^1(\Omega)$. Since $u_\lambda$ is increasing on $\lambda$, it follows that the limit $u^* = \lim_{\lambda \rightarrow \lambda^*} u_\lambda$ exists in $L^1(\Omega)$. Moreover, again by (b), we have $f(u^*) \in L^1(\Omega, \delta(x) dx)$, where $\delta(x)$ represents the distance function of $x$ to $\partial \Omega$. Thus, one easily concludes that $u^*$ is a unique minimal nonnegative weak solution of \eqref{2}. The part of minimality and uniqueness follows directly from Lemma \ref{uniqueness-weak-solutions-lemma} in Section 5.

A fundamental question raised by Brezis and Vázquez \cite{BrezisVazquez1997} is whether the extremal solution $u^*$ is bounded and so classical. There are important answers concerning the case ${\cal L} = \Delta$. More previously, in 1973, Joseph and Lundgren \cite{JosephLundgren1973} showed that $u^*$ is not bounded when $f(u) = e^u$ and $\Omega$ is the unit ball in $\R^n$ for $n \geq 10$, because $\lambda^* = 2(n-2)$ and $u^*(x) = -2 \log |x|$ and, in 1980, Mignot and Puel \cite{MignotPuel1980} proved the boundedness of $u^*$ for any dimension $n \leq 9$. Thenceforth, other quite relevant results were established during the last two decades: In 2000, Nedev \cite{Nedev2000} showed that the extremal solution $u^*$ is bounded in dimensions $n \leq 3$ whenever $f$ is convex; still for such nonlinearities, Cabré and Capella \cite{CabreCapella2006} proved in 2006 that $u^*$ is bounded when $\Omega$ is a ball in $\R^n$ with $n \leq 9$; regarding more general functions $f$, Cabré \cite{Cabre2010} proved in 2010 that $u^*$ is bounded when $n \leq 4$ and $\Omega$ is convex; in 2013, Villegas \cite{Villegas2013} removed the convexity of $\Omega$ in dimension $n = 4$, however, convexity of $f$ was required; lastly, Cabré, Figalli, Ros-Oton and Serra \cite{CFRS2019} established recently that $u^*$ is bounded for $n\leq9$ whenever $f$ is convex and, moreover, the dimension $n = 9$ is optimal in some cases.

Regarding the case of systems with two equations ($m = 2$) involving Laplace operators, the problem \eqref{1} has received special attention in recent decades. For instance, we mention the contributions given by Montenegro \cite{Montenegro2005}, namely, existence and nonexistence of strong solution, existence of extremal solution and stability of minimal solutions, and by Aghajani and Cowan \cite{AghajaniCowan}, Cowan \cite{Cowan2011, Cowan2012}, Cowan and Fazly \cite{CowanFazly2014}, Fazly and Ghoussoub \cite{Fazly2015}, where are investigated the boundedness of extremal solutions.

This work is dedicated to the general case $m \geq 1$. Firstly, we study the existence of a $(m-1)$-hypersurface $\Lambda^*$ decomposing $\R_+^m$ into two connected components ${\cal A}$ and ${\cal B}$ such that \eqref{1} admits a stable minimal positive strong solution whenever $\Lambda \in {\cal A}$ and has no nonnegative strong solution for $\Lambda \in {\cal B}$. If further the coefficients of ${\cal L}_i$ are a little more smooth, then a unique minimal nonnegative weak solution, called extremal solution, exists for every $\Lambda \in \Lambda^*$. We also investigate its boundedness in lower dimensions for Laplace operators and maps $F$ of potential type.

In order to state the first assumptions to be satisfied by the map $F$, we introduce some notations for convenience. For $a=(a_1,\ldots,a_m)$, $b=(b_1,\ldots,b_m) \in\mathbb{R}^m$ and $\gamma=(\gamma_1,\ldots,\gamma_m) \in \R^m_+$, consider the product type notations

\[
ab =(a_1b_1,\ldots,a_mb_m)\, ,\ \ \ \Pi\, \gamma = \Pi_{i=1}^m \gamma_i
\]
and the nonlinear shift

\[
S_\gamma (a)=(|a_2|^{\gamma_1-1}a_2, |a_3|^{\gamma_2-1}a_3, \ldots,|a_m|^{\gamma_{m-1}-1}a_m, |a_1|^{\gamma_m-1}a_1)\, .
\]
We also denote $a < b$ (or $a \leq b$) to mean that $a_i < b_i$ (or $a_i \leq b_i$) for every $i=1,\ldots,m$.

The map $F:\Omega\times \R^m \to \R^m$ of interest here must meet some quite general regularity conditions, namely, $F(x, \cdot): \R^m \to \R^m$ is continuous for $x \in \Omega$ almost everywhere and $F(\cdot, t): \Omega \to \R^m$ belongs to $L^n(\Omega; \R^m):= L^n(\Omega) \times \cdots \times L^n(\Omega)$ for every $t \in \R_+^m$. Moreover, $F$ must satisfy three conditions closely related to the corresponding scalar ones:

\begin{itemize}
\item[(A)] $F(x,0) > 0$ for $x \in \Omega$ almost everywhere ($0 = (0,\ldots,0) \in \R^m$) ($F$ is positive at the origin);
\item[(B)] $F(x,s) \leq F(x,t)$ for $x \in \Omega$ almost everywhere and $s, t \in \R^m$ with $0 \leq s \leq t$ ($F$ is nondecreasing);
\item[(C)] There exist a map $\rho \in L^n(\Omega; \R^m)$ such that $\rho(x) > 0$ for $x \in \Omega$ almost everywhere and a tuple $\alpha \in \R^m_+$ with $\Pi\, \alpha = 1$ such that for any $\kappa > 0$, we have

    \[
    F(x,t) \geq \kappa \rho(x) S_\alpha(t)
    \]
    for $x \in \Omega$ almost everywhere and $t \in \R^m_+$ with $|t| > M$, where $M$ is a positive constant depending on $\kappa$ ($F$ is ``nonlinearly superlinear coupled" at infinity).
\end{itemize}
Observe that (C) provides simultaneously a coupling for \eqref{1} through the lower comparison of $F(\cdot,t)$ by the shift $S_\alpha(t)$ and a nonlinear superlinearity for maps $F$ expressed by the condition $\Pi\, \alpha = 1$.

We highlight below some strongly coupled prototype examples of maps $F$ that fulfill the above hypotheses (A), (B) and (C).

\begin{exam} Consider $F(x,t) = \rho(x) S_\beta \circ {\cal E}(t)$, where $\rho = (\rho_1, \ldots, \rho_m) \in L^n(\Omega; \R^m)$ is a positive map almost everywhere in $\Omega$, $\beta = (\beta_1, \ldots, \beta_m) \in \R^m_+$ is a $m$-tuple and ${\cal E}(t) := (e^{t_1}, \dots, e^{t_m})$. In explicit form, we have

\[
F(x,t) = (\rho_1(x) e^{\beta_1 t_2}, \ldots, \rho_{m-1}(x) e^{\beta_{m-1} t_m}, \rho_m(x) e^{\beta_m t_1}).
\]
\end{exam}

\begin{exam} Consider $F(x,t) = {\cal P}_\alpha \circ (\tau + \rho S_\beta)(x,t)$, where $\rho = (\rho_1, \ldots, \rho_m)$ and $\tau = (\tau_1, \ldots, \tau_m) \in L^n(\Omega; \R^m)$ are positive maps almost everywhere in $\Omega$, $\alpha = (\alpha_1, \ldots, \alpha_m), \beta = (\beta_1, \ldots, \beta_m) \in \R^m_+$ are $m$-tuples such that $\Pi \, \alpha \beta > 1$ and ${\cal P}_\alpha(t) = (t_1^{\alpha_1}, \ldots, t_m^{\alpha_m})$. In explicit form, we have

\[
F(x,t) = ((\rho_1(x) t_2^{\beta_1} + \tau_1(x))^{\alpha_1}, \ldots, (\rho_{m-1}(x) t_m^{\beta_{m-1}} + \tau_{m-1}(x))^{\alpha_{m-1}}, (\rho_m(x) t_1^{\beta_m} + \tau_m(x))^{\alpha_m}).
\]
\end{exam}

\begin{exam} Consider $F(x,t) = S_\beta(\tau(x) + A(x) t)$, where $A = [A_{ij}] \in L^n(\Omega; \R^{m^2})$ is a nonnegative matrix almost everywhere in $\Omega$ with positive diagonal entries, $\tau = (\tau_1, \ldots, \tau_m) \in L^n(\Omega; \R^m)$ is a positive map almost everywhere in $\Omega$ and $\beta = (\beta_1, \ldots, \beta_m) \in \R^m_+$ is a $m$-tuple such that $\Pi \, \beta > 1$. In explicit form, we have

\[
F(x,t) = \left(\left(\sum_{j=1}^m A_{2j}(x) t_j + \tau_2(x)\right)^{\beta_1}, \ldots, \left(\sum_{j=1}^m A_{mj}(x) t_j + \tau_m(x)\right)^{\beta_{m-1}}, \left(\sum_{j=1}^m A_{1j}(x) t_j + \tau_1(x)\right)^{\beta_m}\right).
\]
\end{exam}

\begin{exam} Consider $F(x,t) = \rho(x) \nabla f(t)$, where $\rho \in L^n(\Omega)$ is a positive function almost everywhere in $\Omega$ and $f(t) = \Pi_{i=1}^m f_i(t_i)$, $t = (t_1, \ldots, t_m)$, being $f_i: \R \rightarrow \R$ convex $C^1$ functions satisfying the scalar conditions (a), (b) and (c) and also $f_i'(0) > 0$ for every $i= 1, \ldots, m$.
\end{exam}

Our first result states that

\begin{theo}\label{separation}
Let $F:\Omega\times \R^m \to \R^m$ be a map such that $F(x, \cdot): \R^m \to \R^m$ is continuous for $x \in \Omega$ almost everywhere and $F(\cdot, t): \Omega \to \R^m$ belongs to $L^n(\Omega; \R^m)$ for any $t \in \R^m$. Assume also that $F$ satisfies (A), (B) and (C). Then, there exists a set $\Lambda^*$ decomposing $\R^m_+$ into two sets ${\cal A}$ and ${\cal B}$ satisfying two assertions:
\begin{itemize}
\item[(I)] Problem \eqref{1} admits a minimal positive strong solution $u_\Lambda \in W^{2,n}(\Omega; \R^m) \cap W_0^{1,n}(\Omega; \R^m)$ for any $\Lambda \in {\cal A}$. Moreover, the map $\Lambda \in {\cal A} \mapsto u_\Lambda$ is nondecreasing and continuous;
\item[(II)] Problem \eqref{1} admits no nonnegative strong solution for any $\Lambda \in {\cal B}$.
\end{itemize}
\end{theo}

The next theorem exhibits some properties about regularity, monotonicity and asymptotic behavior of the separating set $\Lambda^*$.

\begin{theo}\label{properties}
Let $F$ be as in Theorem \ref{separation}. The set $\Lambda^*$ admits a parametrization $\Phi : \R^{m-1}_+ \to \Lambda^* \subset \R^m_+$ of the form $\Phi(\sigma) = (\lambda^*(\sigma), \nu^*(\sigma)) := (\lambda^*(\sigma),\lambda^*(\sigma) \sigma)$, where $\lambda^*(\sigma)$ is a positive number for each $\sigma \in \R^{m-1}_+$.
Moreover, we have:
\begin{enumerate}
\item[(I)] $\Phi$ is continuous;
\item[(II)] $\lambda^*(\sigma)$ is nonincreasing, that is, $\lambda^*(\sigma_1) \geq \lambda^*(\sigma_2)$ whenever $0 < \sigma_1 \leq \sigma_2$;
\item[(III)] For each $0 < \sigma_1 \leq \sigma_2$, there exists $i\in\{1,\ldots,m-1\}$ such that $\nu^*_i(\sigma_1)\leq\nu^*_i(\sigma_2)$. In particular, for $m=2$, $\nu^*$ is nondecreasing;
\item[(IV)] $\lambda^*(\sigma) \to 0$ as $\sigma_i \to \infty$ for each fixed $i = 1, \ldots, m-1$ (partial limits).
\end{enumerate}
\end{theo}
As we shall see later, the sets ${\cal A}$ and ${\cal B}$ are related to the parametrization $\Phi(\sigma)$ by

\begin{eqnarray*}
&& {\cal A} = \{(\lambda,\lambda \sigma):\ 0 < \lambda < \lambda^*(\sigma),\ \sigma \in \R^{m-1}_+\},\\
&& {\cal B} = \{(\lambda,\lambda \sigma):\ \lambda > \lambda^*(\sigma),\ \sigma \in \R^{m-1}_+\}.
\end{eqnarray*}
In particular, Theorem \ref{properties} implies that the sets ${\cal A}$ and ${\cal B}$ are connected components of $\R^m_+$.

The first coordinate $\lambda^*(\sigma)$ of $\Phi(\sigma)$ may be bounded or unbounded depending on the growth of the first coordinate $f_1(x,t)$ of the map $F(x,t)$ on the variable $t$. We discuss about this point in the following two remarks:

\begin{remark} $\lambda^*(\sigma)$ is bounded if

\[
f_1(x,t_1,0, \ldots,0) \geq \tau(x) t_1\ {\rm for}\ x \in \Omega\ {\rm a.e.\ and}\ t_1 \geq 0,
\]
where $\tau \in L^n(\Omega)$ and $\tau(x) > 0$ for $x \in \Omega$ almost everywhere. This is a consequence of the strong maximum principle satisfied by the operator ${\cal L}_1$, by assumption. Let $\varphi_1 \in W^{2,n}(\Omega) \cap W_0^{1,n}(\Omega)$ be a positive Dirichlet eigenfunction corresponding to the principal eigenvalue $\mu_1 = \mu_1(-{\cal L}_1, \Omega, \tau)$, that is,

\[
\left\{
\begin{array}{rlllr}
-{\cal L}_1 \varphi_1 &=& \mu \tau(x) \varphi_1 & {\rm in} & \Omega, \\

\varphi_1 &=&0 & {\rm on} & \partial \Omega.
\end{array}\right.
\]
The existence of the couple $(\mu_1, \varphi_1)$ is guaranteed by Theorem 6.4 of \cite{LoGo1996} along with a usual approximation argument. We claim that $\lambda \leq \mu_1$ for every $0 < \lambda < \lambda^*(\sigma)$, so that $\lambda^*(\sigma)$ is bounded above by $\mu_1$ and the assertion follows. In fact, assume by contradiction that $\lambda > \mu_1$ for some $0 < \lambda < \lambda^*(\sigma)$ and set $\Lambda = (\lambda, \lambda \sigma)$. By the above characterization of ${\cal A}$, one has $\Lambda \in {\cal A}$ and so, by Theorem \ref{separation}, \eqref{1} admits a minimal positive strong solution $u_\Lambda = (u_1, \ldots, u_m)$. Consider now the set $S = \{s > 0:\, u_1 > s \varphi_1\ {\rm in}\ \Omega\}$ which, by strong maximum principle and Hopf's lemma, is clearly nonempty and bounded above. Let $s^* = \sup S > 0$. Then, $u_1 \geq s^* \varphi_1$ in $\Omega$ and

\begin{eqnarray*}
-{\cal L}_1 (u_1 - s^* \varphi_1) &=& \lambda f_1(x, u_\Lambda) - s^* \mu_1 \tau(x) \varphi_1 \\
&\geq& \lambda \tau(x) u_1 - s^* \mu_1 \tau(x) \varphi_1 \\
&\geq& s^* (\lambda - \mu_1) \tau(x) \varphi_1 \\
&>& 0\ \ {\rm in}\ \Omega
\end{eqnarray*}
Using again strong maximum principle and Hopf's lemma, we derive the contradiction $u_1 \geq (s^* + \varepsilon) \varphi_1$ in $\Omega$ for $\varepsilon > 0$ small enough.
\end{remark}

\begin{remark} $\lambda^*(\sigma)$ is unbounded if for any $\lambda > 0$ there exists $t_0 = (t_1, \ldots, t_m) \in \R^m_+$ such that

\[
\lambda ||f_1(\cdot, t_0)||_{L^n(\Omega)} \leq t_1\, .
\]
Let a fixed $\lambda_0 > 0$ and consider a constant $C_0 > 0$ such that, for any $w \in W^{2,n}(\Omega) \cap W_0^{1,n}(\Omega)$,

\begin{gather}\tag{1.3} \label{AE}
||w||_{L^\infty(\Omega)} \leq C_0 ||{\cal L}_1 w||_{L^n(\Omega)}\, .
\end{gather}
By assumption, there exists $t_0 = (t_1, \ldots, t_m) \in \R^m_+$ such that

\[
\lambda_0 C_0 ||f_1(\cdot, t_0)||_{L^n(\Omega)} \leq t_1\, .
\]
For each $i = 2, \ldots, m$, let $w_i \in W^{2,n}(\Omega) \cap W_0^{1,n}(\Omega)$ be the strong solution of the problem

\[
    \left\{
\begin{array}{rlllr}
-{\cal L}_i w &=& f_i(x,t_0) & {\rm in} & \Omega, \\

w&=&0 & {\rm on} & \partial \Omega
\end{array}\right.
    \]
Choose $s > 0$ such that $\overline{u}_i:= s w_i \leq t_i$ in $\Omega$ for $i = 2, \ldots, m$. Now let $\overline{u}_1 \in W^{2,n}(\Omega) \cap W_0^{1,n}(\Omega)$ be the strong solution of

\[
    \left\{
\begin{array}{rlllr}
-{\cal L}_1 w &=& \lambda_0 f_1(x,t_1, \overline{u}_2, \ldots, \overline{u}_m) & {\rm in} & \Omega, \\

w&=&0 & {\rm on} & \partial \Omega
\end{array}\right.
    \]
Thanks to the initial assumption and estimate \eqref{AE}, we have $\overline{u}_1 \leq t_1$ in $\Omega$, so that $\overline{u}_0:= (\overline{u}_1, \ldots, \overline{u}_m) \leq t_0$. But this inequality together with the definition of $\overline{u}_i$ and the assumption (B) imply that $\overline{u}_0$ is a positive supersolution of

\[
\left\{
\begin{array}{rlllr}
-{\cal L} u &=& \Lambda_0 F(x,u) & {\rm in} & \Omega, \\

u&=&0 & {\rm on} & \partial \Omega,
\end{array}\right.
\]
where $\Lambda_0 = (\lambda_0, \lambda_0 \sigma) \in \R^m_+$ and $\sigma = (s/\lambda_0, \ldots, s/\lambda_0) \in \R^{m-1}_+$. Then, noting that the zero map $0 = (0, \ldots, 0)$ is a subsolution of the above problem, one concludes that $\Lambda_0 \in {\cal A}$, by the super-subsolution method (see the proof of Theorem \ref{separation} in Section 3), so that $\lambda^*(\sigma) \geq \lambda_0$. Since $\lambda_0 > 0$ is an arbitrary number, it follows the desired conclusion.
\end{remark}

For the existence of solution on $\Lambda^*$ we require that the coefficients $a^i_{kl}$ and $b^i_j$ of ${\cal L}_i$ belong respectively to $C^2(\overline{\Omega})$ and $C^1(\overline{\Omega})$ for every $i, j, k, l$. Before, we present the notion of weak solution for \eqref{1} (see Definition 1 of \cite{BCMR}).

\begin{defi} Let $\Lambda = (\lambda_1, \ldots, \lambda_m) \in \R_+^m$. We say that $u \in L^1(\Omega; \R^m)$ is a nonnegative weak solution of \eqref{1}, if $u \geq 0$ almost everywhere in $\Omega$,
\begin{gather}\tag{1.4}\label{weak_solution_def_1}
F(\cdot,u(\cdot))\delta(\cdot) \in L^1(\Omega; \R^m)
\end{gather}
and

\[
-\int_\Omega u {\cal L}^* \zeta dx = \Lambda \int_\Omega F(x,u)\zeta dx
\]
for all $\zeta \in W^{2,n}(\Omega; \R^m) \cap W_0^{1,n}(\Omega; \R^m)$ such that ${\cal L}^* \zeta \in L^\infty(\Omega; \R^m)$.
\end{defi}

The next result establish the existence of a unique minimal nonnegative weak solution, called extremal solution.

\begin{theo}\label{fraca}
Let $F:\Omega\times \R^m \to \R^m$ be a map such that $F(x, \cdot): \R^m \to \R^m$ is continuous for $x \in \Omega$ almost everywhere and $F(\cdot, t): \Omega \to \R^m$ belongs to $L^n(\Omega; \R^m)$ for every $t \in \R^m$. Assume that $F$ satisfies (A), (B) and (C) and also $a^i_{kl} \in C^2(\overline{\Omega})$ and $b^i_j \in C^1(\overline{\Omega})$ for every $i, j, k, l$. Then, for any $\Lambda \in \Lambda^*$, Problem \eqref{1} admits an extremal solution $u^*_\Lambda$.
\end{theo}

We now introduce the concept of linearized stability for strong solutions of \eqref{1} which can be seen as steady states of the parabolic problem

\begin{gather}\tag{1.5} \label{3}
\left\{
\begin{array}{rlllr}
\displaystyle{\frac{\partial u}{\partial t}} -{\cal L} u &=& \Lambda F(x,u) & {\rm in} & \Omega \times (0, \infty), \\
u(x,0) &=& u_0(x) & {\rm for} & \ x \in \Omega,\\
u(x,t)&=&0 & {\rm for} & (t,x) \in \partial \Omega \times (0, \infty),
\end{array}\right.
\end{gather}
where $u(x,t) = (u_1(x,t), \ldots, u_m(x,t))$ and  $F(x,t) = (f_1(x,t), \ldots, f_m(x,t))$.

Let $u$ be a solution of the system in \eqref{3}. For a suitable notion of stability of $u$, one hopes naturally that for any globally defined strong solution of the form $v(x,t) = u(x) + e^{-\eta t} \varphi(x)$, one necessarily has $\eta > 0$. Assuming that $F(x,\cdot)$ is of $C^1$ class for $x \in \Omega$ almost everywhere, writing the system for $v(x,t)$ and letting $t \rightarrow \infty$, we get an eigenvalue problem under the form of elliptic system, namely,

\begin{gather}\tag{1.6} \label{4}
\left\{
\begin{array}{rlllr}
-{\cal L} \varphi - \Lambda (A(x,u) \varphi) &=& \eta \varphi & {\rm in} & \Omega, \\
\varphi &=& 0 & {\rm on}& \partial \Omega,
\end{array}\right.
\end{gather}
where $A(x,t)$ denotes the matrix $[A_{ij}(x,t)]$ with entries

\[
A_{ij}(x,t) := \frac{\partial f_i}{\partial t_j}(x,t)\, .
\]

\begin{defi}\label{stability-definition}
A steady state $u$ of the problem \eqref{3} is said to be stable in the linearized sense, if the system \eqref{4} has a nonnegative principal eigenvalue $\eta_1$. In addition, in case $\eta_1$ is positive, we say that $u$ is asymptotically stable.
\end{defi}

An underlying question is first whether \eqref{4} admits a principal eigenvalue. Its existence is only ensured under additional conditions on the matrix $A(x,t)$ such as:

\begin{gather}\tag{1.7} \label{5}
\sup_{|t| \leq a}|A(x,t)| \in L^n(\Omega)\ \ {\rm for \ every} \ a>0
\end{gather}
and

\begin{itemize}
\item[(D)] $A(x,u(x))$ is irreducible for every $x \in \Omega$ whenever $u$ is positive in $\Omega$.
\end{itemize}

Clearly, the assumptions (B) and (D) imply that each matrix $A(x,u(x))$ is cooperative and irreducible, provided that $u$ is positive. On the other hand, cooperativity and irreducibility are well-known conditions in the literature that yield the existence of a principal eigenvalue $\eta_1$, see for example the monograph \cite{Am} and various references therein.

It also deserves mention that (D) holds for example when $A_{ij}(x,t) \geq 0$ for all $x \in \Omega$ a.e., $t > 0$ and $i \neq j$ and, in addition, $|\{ x \in \Omega:\ A_{is_i}(x,t) > 0, \forall t > 0 \}| >0$ for all $i$, where $s$ is the shift permutation of $\{1, \ldots, m\}$, that is, $s_i = i+1$ for $i = 1, \ldots, m-1$ and $s_m = 1$. Indeed, we have $A_{is_i}(\cdot,u(\cdot)) \not\equiv 0$ for all $i$ and so the irreducibility of $A(x,u(x))$ readily follows by using that $s$ is a permutation of order $m$. Besides, in several examples, the nonlinear shift $S_\alpha$ and the shift permutation $s$ are strongly connected.

\begin{theo}\label{stability}
Let $F:\Omega\times \R^m \to \R^m$ be a map such that $F(x, \cdot): \R^m \to \R^m$ is of $C^1$ class uniformly on $x \in \Omega$ and the second variable in compacts of $\R^m$, $F(\cdot, t): \Omega \to \R^m$ belongs to $L^n(\Omega; \R^m)$ for every $t \in \R^m$ and \eqref{5} is satisfied. Assume also that $F$ satisfies (A), (B), (C) and (D). Then, for any $\Lambda \in {\cal A}$, the minimal positive strong solution $u_\Lambda$ is stable. Moreover, if $F$ is convex and $a^i_{kl} \in C^2(\overline{\Omega})$ and $b^i_j \in C^1(\overline{\Omega})$ for every $i, j, k, l$, then $u_\Lambda$ is asymptotically stable.
\end{theo}

\noindent Here, by a $C^1$ map $F(x,t)$ on the variable $t$ uniformly on $x \in \Omega$ and $t$ in compacts of $\R^m$, we mean that $F(x,t)$ is of $C^1$ class on $t \in \R^m$ for each fixed $x \in \Omega$ and, in addition, for any compact subset $K \subset \R^m$,

\[
F(x,t) = F(x,t_0) + A(x,t_0)(t-t_0) + r(x,t,t_0),
\]
where

\[
\lim_{t \rightarrow t_0} \frac{r(x,t,t_0)}{|t - t_0|} = 0
\]
uniformly on $x \in \Omega$ and $t_0 \in K$.

Theorems \ref{separation}, \ref{properties}, \ref{fraca} and \ref{stability} provide substantial improvements in the case $m = 2$, previously studied in \cite{Montenegro2005}, and also new contributions for $m \geq 3$ under the milder assumptions (A), (B), (C) and (D).

Finally, we wish to discuss the boundedness of extremal solutions $u^*_{\Lambda}$ of \eqref{1} for any $m \geq 1$ in the case that all ${\cal L}_i$ are Laplacian and $F: \R^m \rightarrow \R^m$ is a potential field. More precisely, given a function $f:\mathbb{R}^m\rightarrow\mathbb{R}$ of $C^2$ class, we consider the problem for the gradient field $F = \nabla{f}$:

\begin{gather}\tag{1.8}\label{gradient-system-intro}
\left\{
\begin{array}{rlllr}
-{\Delta} u &=& \Lambda \nabla f(u) & {\rm in}& \Omega, \\
u&=&0 & {\rm on} & \partial \Omega.
\end{array}\right.
\end{gather}
The system \eqref{gradient-system-intro} is called symmetric in the sense that the matrices $A(x,t) = \Hessian f(t)$ associated to the linearized problem are symmetric.

For $m=2$, Cowan and Fazly \cite{CowanFazly2014} showed that the extremal solutions $u^*_{\Lambda}$ with $\Lambda \in \Lambda^*$ are bounded for convex domains in dimensions $n=2$ and $n=3$ and for functions $f$ of separable variables, that is, $f(u_1,u_2)=f_1(u_1)f_2(u_2)$, where $f_1$ and $f_2$ satisfy some conditions that lead to (A), (B), (C) and (D). When $\Omega$ is a ball in $\mathbb{R}^n$, the authors also proved the boundedness of corresponding extremals provided that $2 \leq n \leq 9$. Previous results involving regularity of extremal solutions can be found in \cite{Cowan2011}.

When $m\geq3$ and $\Omega = B$ is the unit ball, Fazly \cite{Fazly2015} established for symmetric systems that stable minimal classical solutions $u_\Lambda = (u_1, \ldots, u_m)$, which are radial for any $\Lambda = (\lambda_1, \ldots, \lambda_m) \in \mathcal{A}$, satisfy

\begin{itemize}

\item[(i)] $\sum_{i=1}^m \frac{1}{\sqrt{\lambda_i}} |u_i(r)| \leq C_{n,m}\sum_{i=1}^m \frac{1}{\sqrt{\lambda_i}}\|u_i\|_{H^1(B\setminus\overline{B}_{1/2})}$ for $r \in (0,1]$, if $2 \leq n \leq 9$;

\item[(ii)] $\sum_{i=1}^m \frac{1}{\sqrt{\lambda_i}} |u_i(r)| \leq C_{n,m}(1+|\log{r}|)\sum_{i=1}^m \frac{1}{\sqrt{\lambda_i}}\|u_i\|_{H^1(B\setminus\overline{B}_{1/2})}$ for $r \in (0,1]$, if $n=10$;

\item[(iii)] $\sum_{i=1}^m \frac{1}{\sqrt{\lambda_i}} |u_i(r)| \leq C_{n,m}r^{-\frac{n}{2}+\sqrt{n-1}+2}\sum_{i=1}^m \frac{1}{\sqrt{\lambda_i}}\|u_i\|_{H^1(B\setminus\overline{B}_{1/2})}$ for $r\in(0,1]$, if $n\geq 11$

\end{itemize}
for some positive constant $C_{n,m}$ independent of $r$, where $B_\delta$ denotes the ball in $\R^n$ of radius $\delta$ centered at the origin. These three estimates were previously obtained for $m = 1$ by Villegas \cite{Villegas2012}.

On the other hand, a straightforward argument produces an estimate for the above terms $\|u_i\|_{H^1(B\setminus\overline{B}_{1/2})}$ regarding $\Lambda = (\lambda_1, \ldots, \lambda_m) \in {\cal A}$ near the extremal set $\Lambda^*$. In fact, let $\varphi\in C^{\infty}(\mathbb{R}^n)$ be a fixed function such that $0\leq\varphi\leq1$, $\varphi=0$ in $B_{1/4}$ and $\varphi=1$ in $B\setminus\overline{B}_{1/2}$. Multiplying the $i$th equation of \eqref{gradient-system-intro}, $-\Delta u_i=\lambda_i f_{u_i}(u_\Lambda)$, by $u_i\varphi^2$, integrating by parts and using Young's inequality and that $u_i$ is radially decreasing, we derive

\begin{eqnarray*}
\int_{B\setminus \overline{B}_{1/4}}|\nabla u_i|^2\varphi^2 dx & = &  \lambda_i\int_{B\setminus \overline{B}_{1/4}} f_{u_i}(u_\Lambda)u_i\varphi^2dx + \int_{\partial B_{1/4}}\dfrac{\partial u_i}{\partial \nu}u_i\varphi^2d\sigma - \int_{B\setminus \overline{B}_{1/4}} 2u_i\varphi\nabla u_i\cdot\nabla\varphi dx\\
& \leq & \lambda_i\int_{B\setminus \overline{B}_{1/4}} f_{u_i}(u_\Lambda)u_idx + 2\int_{B\setminus \overline{B}_{1/4}} |\nabla\varphi|^2u^2_idx + \dfrac{1}{2}\int_{B\setminus \overline{B}_{1/4}} |\nabla u_i|^2\varphi^2dx.
\end{eqnarray*}
Thus,

\begin{gather}\tag{1.9}\label{ball-1}
\int_{B\setminus\overline{B}_{1/2}}|\nabla u_i|^2dx \leq  2\lambda_i\int_{B\setminus \overline{B}_{1/4}} f_{u_i}(u_\Lambda)u_idx + C\int_{B\setminus \overline{B}_{1/4}} u^2_idx.
\end{gather}
Moreover, we also have
\begin{gather}\tag{1.10}\label{ball-2}
\sup_{B\setminus \overline{B}_{1/4}} u_i(x)=u_i(1/4)\leq \dfrac{1}{|B_{1/4}\setminus \overline{B}_{1/8}|}\int_{B_{1/4}\setminus \overline{B}_{1/8} }u_i(x)dx\leq C_n\|u_i\|_{L^1(B)}\leq C_n\|u_\Lambda\|_{L^1(B;\mathbb{R}^m)}.
\end{gather}
Therefore, for a fixed $\sigma \in \R^{m-1}_+$, Lemma \ref{boundedness-of-u} in Section 5 and (\ref{ball-1}) and (\ref{ball-2}) provide, for any $\frac{1}{2}\lambda^*(\sigma) < \lambda < \lambda^*(\sigma)$,

\[
\|u_\Lambda\|_{H^1(B\setminus \overline{B}_{1/2}; \R^m)} \leq C
\]
for some positive constant $C$ independent of $\lambda$. In particular, if $(\lambda_k)$ is an increasing sequence converging to $\lambda^*(\sigma)$, then the above estimate holds for $\Lambda_k = (\lambda_k, \lambda_k \sigma)$ with $k$ large enough. Hence, plugging the radial solutions $u_{\Lambda_k}$ in the above inequalities (i), (ii) and (iii) and after letting $k \rightarrow \infty$, we deduce the following result for the extremal solution $u^*_\Lambda$:

\begin{theo}
Let $F:\mathbb{R}^m\rightarrow\mathbb{R}^m$ be a $C^1$ potential field satisfying (A), (B), (C) and (D). If $\Omega = B$, then the extremal solution $u_{\Lambda}^*$ of (\ref{gradient-system-intro}) is radial for any $\Lambda \in \Lambda^*$ and, in addition, satisfies

\begin{itemize}

\item[(I)] $u_{\Lambda}^* \in L^{\infty}(B)$, if $2 \leq n \leq 9$;

\item[(II)] $u_{\Lambda}^*(r) \leq C(1+|\log{r}|)$ for $r\in(0,1]$, if $n=10$;

\item[(III)] $u_{\Lambda}^*(r) \leq Cr^{-\frac{n}{2}+\sqrt{n-1}+2}$ for $r\in(0,1]$, if $n \geq 11$,
\end{itemize}
where $C > 0$ is a constant depending only on $n$, $m$ and $\Lambda$. In particular, $u^*_{\Lambda}$ is bounded whenever $2 \leq n \leq 9$.
\end{theo}
Inspired on the insightful developing done in the scalar context by Cabré in \cite{Cabre2010}, we also study the regularity of extremal solutions of (\ref{gradient-system-intro}) for $C^{1,1}$ domains.

In a precise way:

\begin{theo}\label{teo-extremal-strong}
Let $F:\mathbb{R}^m\rightarrow\mathbb{R}^m$ be a $C^1$ potential field satisfying (A), (B), (C) and (D). Assume that $\Omega$ is convex and $n=2$ or $n=3$. Then, for any $\Lambda \in \Lambda^*$, the extremal solution $u_{\Lambda}^*$ of (\ref{gradient-system-intro}) is bounded.
\end{theo}
Theorem \ref{teo-extremal-strong} fairly improves the result established by Cowan and Fazly for potential fields with $m = 2$ (see Theorem 2.1 of \cite{CowanFazly2014}) and is a novelty for $m \geq 3$.

The paper is organized into six sections. In Section 2 we investigate principal eigenvalues for strongly coupled nonlinear elliptic systems and obtain a key ingredient, namely, principal spectral hypersurfaces associated to such systems. Section 3 is dedicated to the proof of Theorem \ref{separation} which provides an extremal separating set with respect to existence of positive strong solutions. Section 4 is devoted to the proof of Theorem \ref{properties} which exhibits some qualitative properties to this set. In Section 5 we prove Theorem \ref{fraca} which ensures the existence of extremal solution on the separating set. In Section 6 we prove the stability of minimal positive strong solutions stated in Theorem \ref{stability}. Finally, in Section 7 we prove Theorem \ref{teo-extremal-strong} which guarantees the regularity of extremal solutions on convex domains in $\R^n$ for dimensions $n = 2, 3$.

\section{The principal spectral hypersurface}

Consider the nonlinear eigenvalue problem

\begin{gather}
\left\{\begin{array}{rlllr}\tag{2.1} \label{FracLE-eign}
-{\cal L} \varphi &=&\Lambda\rho(x)S_\alpha \varphi &\text{ in }& \Omega, \cr
\varphi&=&0 &\text{ on }& \partial \Omega,
\end{array}\right.
\end{gather}
where $\rho \in L^{n}(\Omega; \R^m)$ is a map satisfying $\rho(x) > 0$ for $x \in \Omega$ almost everywhere, $m \geq 1$ and $\alpha \in \R_+^m$ satisfies $\Pi\, \alpha = 1$.

The tuple $\Lambda\in\mathbb{R}^m$ is said to be an eigenvalue of \eqref{FracLE-eign} if the problem admits a nontrivial strong solution $\varphi \in W^{2,n}(\Omega; \R^m) \cap W^{1,n}_0(\Omega; \R^m)$ which by Sobolev embedding is in $C_0^1(\overline{\Omega}; \R^m)$. Here $C_0^1(\overline{\Omega}; \R^m)$ stands for the Banach space $\{u \in C^1(\overline{\Omega}; \R^m): u= 0\ {\rm on}\ \partial \Omega\}$ endowed with the norm

\[
\| u \|_{C^1(\overline{\Omega}; \R^m)} := \sum_{i=1}^m \| u_i \|_{C^1(\overline{\Omega})},
\]
where $C^1(\overline{\Omega}; \R^m)$ denotes the product space $C^1(\overline{\Omega}) \times \cdots \times C^1(\overline{\Omega})$. Furthermore, if all components of $\varphi$ are positive in $\Omega$, then $\Lambda$ is called a principal eigenvalue of \eqref{FracLE-eign}. We denote by $\Gamma_\alpha(\Omega, \rho)$ the set of all principal eigenvalues of the above problem.

The central point to be addressed in this section is the characterization of the set $\Gamma_\alpha(\Omega, \rho)$. In particular, we will show that this set makes up a smooth hypersurface in $\R^m_+$ which will be referred as principal hypersurface of $\R^m$, basically for two reasons. Firstly, the above problem is an extension of the Dirichlet eigenvalue problem for uniformly elliptic operators which corresponds to $m = 1$ and $\rho = 1$ in $\Omega$. So, in this case, the set $\Gamma_\alpha(\Omega, \rho)$ is single and represents the principal eigenvalue of ${\cal L}_1$. Secondly, when $m = 2$ it has been proved in \cite{Montenegro2000} that the set $\Gamma_\alpha(\Omega, \rho)$ is the image of a smooth curve satisfying some important analytical properties.

Strong maximum principle and Hopf's lemma for uniformly elliptic operators, here denoted shortly by {\bf (SMP)} and {\bf (HL)}, will play a key role in our proofs.

In the sequel, we will use a nonlinear version of Krein-Rutman Theorem (see for example Chang \cite{Chang2009} or Mahadevan \cite{Mahadevan2007}) to show that the set $\Gamma_\alpha(\Omega, \rho)$ can be seen as the inverse image of a regular value by a smooth function $H:\R^m_+\to\mathbb{R}$.

\begin{prop}\label{Lane-Emden}
Let $\rho \in L^n(\Omega; \R^m)$ be a positive map almost everywhere in $\Omega$ and $\alpha = (\alpha_1, \ldots, \alpha_m) \in \R_+^m$ be a tuple satisfying $\Pi\, \alpha = 1$ with $m \geq 1$. Then, there exist a smooth function $H: \R_+^m\to\mathbb{R}$ and a positive constant $\lambda_*$ such that

\[
\Gamma_\alpha(\Omega, \rho)= \{\Lambda \in \R_+^m :\, H(\Lambda)=\lambda_*\}\, .
\]
More specifically, the function $H$ is given by

\[
H(\Lambda)= H(\lambda_1, \ldots, \lambda_m) := \lambda_1 \lambda^{\alpha_1}_2 \ldots \lambda^{\alpha_1\ldots\alpha_{m-1}}_m\, .
\]
\end{prop}

\begin{proof}

Consider the Banach spaces

\[
E_0 = C_0(\overline{\Omega}):= \{u \in C(\overline{\Omega}):\ u= 0\ {\rm on}\ \partial \Omega\},
\]

\[
E_1 = C_0^1(\overline{\Omega}):= \{u \in C^1(\overline{\Omega}):\ u= 0\ {\rm on}\ \partial \Omega\}
\]
endowed respectively with the usual norms $\| \cdot \|_{E_0} := \| \cdot \|_{C(\overline{\Omega})}$ and $\| \cdot \|_{E_1} := \| \cdot \|_{C^1(\overline{\Omega})}$.

For each $i=1, \ldots, m$, denote by $T_i: E_0 \rightarrow E_1$ the operator defined by $T_i u = v$, where $v \in W^{2,n}(\Omega) \cap W_0^{1,n}(\Omega) \subset E_1$ is the unique strong solution of the problem

\begin{gather}\tag{2.2}\label{EqOperador}
\left\{\begin{array}{rlllr}
-{\cal L}_i v &=& \rho_i(x) |u|^{\alpha_i - 1} u &\text{ in }& \Omega, \\
v&=&0 &\text{ on }& \partial \Omega. \end{array}\right.
\end{gather}
By the $L^n$ Calderón-Zygmund theory, each map $T_i$ is well-defined and continuous.

Let $J$ be the inclusion from $E_1$ into $E_0$ which is clearly compact. We now define the composition operator $T: E_0 \rightarrow E_0$ by $T:= T_1\circ\cdots\circ T_m \circ J$. Notice that $T$ is continuous and compact. Moreover, the existence and uniqueness of solution of \eqref{EqOperador} yield the $\alpha_i$-homogeneity property of $T_i$, that is,

\[
T_i(\tau u) = \tau^{\alpha_i} T_i(u)
\]
for all $\tau > 0$ and $u \in E_0$. So, by using that $\Pi\, \alpha = 1$, it follows that $T$ is positively $1$-homogeneous.

Consider now the positive cone of $E_j$ given by $K_j = \{u \in E_j:u\geq 0\ {\rm in}\ \Omega \}$ for $j=0,1$, so $E_j$ is an ordered Banach space with respect to $K_j$. We assert that $T$ is strongly monotone with respect to $K_0$. In fact, the positivity of $\rho_i$ and {\bf (SMP)} imply that each operator $T_i$ is monotone, so that $T$ is monotone too. Moreover, it follows that $T_i u_2 - T_i u_1 \in K_0 \setminus \{0\}$ whenever $u_2 - u_1 \in K_0 \setminus \{0\}$. Consequently, $\tilde{u}_2 - \tilde{u}_1 \in K_0 \setminus \{0\}$ provided that $u_2 - u_1 \in K_0 \setminus \{0\}$, where $\tilde{u} = T_2\circ\cdots\circ T_m (u)$. Thus, {\bf (HL)} leads to $T_1 \tilde{u}_2 - T_1 \tilde{u}_1 \in \mathring{K}_1$, where $\mathring{K}_1$ denotes the interior of $K_1$. In other words, $T u_2 - T u_1 \in \mathring{K}_1 \subset \mathring{K}_0$ and the claim follows. Therefore, the positively $1$-homogeneous operator $T$ fulfills all assumptions required by the nonlinear Krein-Rutman theorem. More specifically, by its fairly complete version provided in Theorem 1.4 of \cite{Chang2009}, there exists a number $\lambda_* > 0$ such that $\lambda^{-1}_*$ is the unique principal eigenvalue of $T$ and is the largest among all real eigenvalues of $T$. Moreover, all eigenfunctions of $T$ associated to the eigenvalue $\lambda^{-1}_*$ are multiple of a fixed eigenfunction $\varphi_* \in \mathring{K}_1$.

We are ready to conclude the proof.

Let $\Lambda = (\lambda_1, \ldots, \lambda_m) \in \Gamma_\alpha(\Omega, \rho) \subset \R_+^m$ and $\varphi = (\varphi_1, \ldots, \varphi_m) \in C_0^1(\overline{\Omega}; \R^m)$ be a positive eigenfunction associated to $\Lambda$. Then, from each equation of the problem \eqref{FracLE-eign}, we deduce that $\varphi_i = \lambda_i T_i \varphi_{i+1}$ for $i=1, \ldots, m-1$ and $\varphi_m = \lambda_m T_m \varphi_1$. Arguing step to step with replacement, we derive

\[
\varphi_1 = \lambda_1 \lambda^{\alpha_1}_2 \cdots \lambda^{\alpha_1\ldots\alpha_{m-1}}_m T_1\circ\cdots\circ T_m(\varphi_1) = H(\Lambda) T \varphi_1\, .
\]
Since $\varphi_1 \in K_0 \setminus \{0\}$, the above equality implies that $H(\Lambda)^{-1}$ is a principal eigenvalue of $T$, so that $H(\Lambda) = \lambda_*$ by uniqueness.

Conversely, let $\Lambda = (\lambda_1, \ldots, \lambda_m) \in \R_+^m$ be such that $H(\Lambda)=\lambda_*$. Choose an eigenfunction $\varphi_* \in K_0$ of $T$ associated to the principal eigenvalue $\lambda^{-1}_*$. Define by recurrence the functions $\varphi_m = \lambda_m T_m \varphi_*$ and $\varphi_i = \lambda_{i} T_{i} \varphi_{i+1}$ for $i=1, \ldots, m-1$. Then, $\varphi = (\varphi_1, \ldots, \varphi_m) \in C_0^1(\overline{\Omega}; \R^m)$ and, by {\bf (SMP)}, the map $\varphi$ is positive in $\Omega$. On the other hand, we also have

\[
\varphi_1 = \lambda_1 \lambda^{\alpha_1}_2 \ldots \lambda^{\alpha_1\ldots\alpha_{m-1}}_m T_1\circ\cdots\circ T_m(\varphi_*) = H(\Lambda) T \varphi_*\, .
\]
Finally, using that $H(\Lambda)=\lambda_*$ and $\varphi_*$ is an eigenfunction of $T$ corresponding to $\lambda^{-1}_*$, the above equality yields $\varphi_1 = \varphi_*$ and this proves that $\Lambda$ is a principal eigenvalue of \eqref{EqOperador}. In other words, we conclude that $\Lambda \in \Gamma_\alpha(\Omega, \rho)$.
\end{proof}

The principal $(m-1)$-hypersurface $\{\Lambda \in \R_+^m :\, H(\Lambda)=\lambda_*\}$ will be naturally represented by the notation $\Lambda_1$. Notice that $\Lambda_1$ decomposes $\R^m_+$ into two unbounded connected components.

An immediate consequence of Proposition \ref{Lane-Emden} is

\begin{corol}\label{Autovalor}
For each vector $\sigma=(\sigma_1,\ldots,\sigma_{m-1})\in\R^{m-1}_+$, there exists a unique number $\theta_*=\theta_*(\sigma)>0$ such that

\[
(\theta_*,\theta_*\sigma_1,\ldots,\theta_*\sigma_{m-1}) \in \Lambda_1\, .
\]
Moreover, we have

\[
\theta_*(\sigma)=\left(\displaystyle{\frac{\displaystyle{\lambda_*}}{\displaystyle{\prod_{i=1}^{m-1}\sigma_i^{\prod_{k=1}^i\alpha_k}}}}\right)^{\displaystyle{\frac{1}{\displaystyle{\sum_{i=1}^{m}\prod_{k=1}^i\alpha_k}}}}.
\]
\end{corol}
In most of this work, we will work with $\Lambda_1$ by means of the parametrization $\sigma \in \R_+^{m-1} \mapsto (\theta_*(\sigma), \theta_*(\sigma) \sigma)$ described in this corollary.\\

\section{The construction of the extremal set $\Lambda^*$}
Throughout paper it will be assumed at least that $F(x, \cdot)$ is continuous for $x \in \Omega$ almost everywhere, $F(\cdot, t) \in L^n(\Omega; \R^m)$ for every $t \in \R_+^m$ and $F$ satisfies (A), (B) and (C).

This section is devoted to the study about existence of a $(m-1)$-hypersurface $\Lambda^*$ decomposing $\mathbb{R}^m_+$ into two sets related to existence and nonexistence of positive strong solution of \eqref{1}. The eigenvalue problem \eqref{FracLE-eign} to $m$-parameters considered in the previous section will play a key role in the construction of $\Lambda^*$.

We start by defining the set

\[
\mathcal{A}_0 = \{\Lambda\in\mathbb{R}_+^m:\; \eqref{1}\text{ admits a positive strong solution }u\in W^{2,n}(\Omega; \R^m) \cap W_0^{1,n}(\Omega; \R^m)\}.
\]

The following lemma makes use of (A) and (B) and implies that $\mathcal{A}_0$ is non-empty.

\begin{lem}\label{existence_lema1}
The inclusion $(B_{\varepsilon_0}(0)\cap \mathbb{R}^m_+)\subset \mathcal{A}_0$ holds for $\varepsilon_0 > 0$ sufficiently small.
\end{lem}

\begin{proof}
A basic fact to be used in the proof is that $F(\cdot, u) \in L^n(\Omega; \R^m)$ whenever $u \in L^\infty(\Omega; \R^m)$. Let $u_k$ be defined recursively by $u_1=0$ and $u_{k+1}=\Lambda(-\mathcal{L} )^{-1}( F(x,u_k))$ for $k \geq 1$. Thanks to the $L^n$ Calderón-Zygmund theory, $u_k$ is well defined and belongs to $ W^{2,n}(\Omega; \R^m) \cap W_0^{1,n}(\Omega; \R^m)$ for every $k \geq 1$.

Fix $t_0\in\mathbb{R}^m_+$ and choose $\varepsilon_0>0$ such that $\varepsilon_0(-\mathcal{L} )^{-1}( F(\cdot,t_0))\leq t_0$ in $\Omega$. For $\Lambda\in B_{\varepsilon_0}(0)\cap\mathbb{R}^m_+$, the condition (B) gives

\[
u_2(x) = \Lambda(-\mathcal{L} )^{-1}( F(x,0)) \leq \varepsilon_0(-\mathcal{L} )^{-1}( F(x,t_0))\leq t_0\ {\rm for}\ x \in \Omega.
\]
In addition, (A) and {\bf (SMP)} imply that $u_2 > 0 = u_1$ in $\Omega$. Using again (B) and {\bf (SMP)}, we derive

\[
u_3(x) = \Lambda(-\mathcal{L} )^{-1}( F(x,u_2)) \leq \varepsilon_0(-\mathcal{L} )^{-1}( F(x,t_0))\leq t_0\ {\rm for}\ x \in \Omega
\]
and

\[
u_3(x) = \Lambda(-\mathcal{L} )^{-1}( F(x,u_2)) \geq  \Lambda(-\mathcal{L} )^{-1}( F(x,u_1)) = u_2(x)\ {\rm for}\ x \in \Omega.
\]
In resume, we have $u_1 \leq u_2 \leq u_3 \leq t_0$ in $\Omega$. Proceeding inductively with the aid of (B) and {\bf (SMP)}, we conclude that the sequence $(u_k)$ is pointwise nondecreasing and uniformly bounded above by $t_0$. Thus, $u_k$ converges pointwise and in $L^q(\Omega; \R^m)$ to $u$ for any $q \geq 1$. Again invoking the Calderón-Zygmund elliptic theory, it follows that $u_k$ converges to $u$ in $W^{2,n}(\Omega;\mathbb{R}^m)$ and so $\Lambda \in \mathcal{A}_0$.
\end{proof}

The idea behind the construction of $\Lambda^*$ lies in considering points on each half-line in $\R^m_+$ starting from the origin at the direction $(1,\sigma)$, that is, $\mathcal{T}_\sigma=\{\lambda>0:\;(\lambda,\lambda\sigma)\in \mathcal{A}_0\}$, where $\sigma \in \R^{m-1}_+$. By Lemma \ref{existence_lema1}, the set $\mathcal{T}_\sigma$ is nonempty for every $\sigma\in\mathbb{R}^{m-1}_+$.

We assert that the segment $\mathcal{T}_\sigma$ is bounded for each fixed $\sigma\in\mathbb{R}^{ m-1}_+$. The next result will be useful in the proof of this claim and will demand the three conditions (A), (B) and (C).

\begin{lem}\label{super}
There exist a positive map $\rho_0 \in L^n(\Omega; \R^m)$, a tuple $\alpha \in \R^m_+$ with $\Pi \alpha = 1$ and a constant $C_0 > 0$ such that, for $x \in \Omega$ almost everywhere and $t \in \R^m_+$,

\[
C_0 F(x,t) \geq \rho_0(x) S_\alpha(t).
\]
\end{lem}

\begin{proof}
Applying (C) for $\kappa = 1$, we get a positive map $\rho \in L^n(\Omega; \R^m)$, a tuple $\alpha \in \R^m_+$ with $\Pi\, \alpha = 1$ and a constant $M > 0$ such that

\[
F(x,t) \geq \rho(x) S_\alpha(t)
\]
for $x \in \Omega$ almost everywhere and $t \in \R^m_+$ with $|t| > M$. On the other hand, by continuity and compactness, there exists a constant $D_0 > 0$ so that

\[
S_\alpha(t) \leq D_0 (1, \ldots, 1)
\]
for all $t \in \R^m_+$ with $|t| \leq M$. Combining this conclusion with the assumptions (A) and (B), we derive

\[
D_0 F(x,t) \geq D_0 F(x,0) \geq F(x,0) S_\alpha(t)
\]
for $x \in \Omega$ almost everywhere and $t \in \R^m_+$ with $|t| \leq M$. Thus, the proof follows by choosing the positive constant $C_0 = \max\{1, D_0\}$ and the positive map $\rho_0(x) = \min\{\rho(x), F(x,0)\}$.

\end{proof}

The main result of this section is as follows:

\begin{lem}\label{existence_lema2}
The set $\mathcal{T}_{\sigma}$ is bounded for every $\sigma \in \R_+^{m-1}$.
\end{lem}
\begin{proof}
Let fixed $\sigma \in \R_+^{m-1}$ and $\Lambda=(\lambda,\lambda\sigma) \in \mathcal{A}_0$. By the definition of $\mathcal{A}_0$, \eqref{1} admits a positive strong solution $u$. Invoking Lemma \ref{super}, we get

\[
-\mathcal{L}u\geq \frac{1}{C_0} \Lambda\rho_0(x)S_\alpha(u),
\]
where $\rho_0 \in L^n(\Omega; \R^m)$ is a positive map and $\alpha \in \R^m_+$ is a tuple such that $\Pi\, \alpha = 1$. Applying Corollary \ref{Autovalor} to \eqref{EqOperador} with weight $\rho_0$, we obtain a principal eigenvalue $\Lambda_0 = (\theta_*,\theta_*\sigma) \in \R^m_+$ and a corresponding positive eigenfunction $\varphi_0$. We claim that $\Lambda \leq C_0 \Lambda_0$. Set $\hat{\alpha}_i=\Pi_{k=i}^m\alpha_k$ and denote $\hat{\alpha}=(\hat{\alpha}_1,\ldots,\hat{\alpha}_m)$ and $s^{\hat{\alpha}}=(s^{\hat{\alpha}_1},\ldots,s^{\hat{\alpha}_m})$. We now consider the set $\mathcal{S}=\{s:\; u > s^{\hat{\alpha}}\varphi_0 \text{ in } \Omega\}$, which is nonempty by {\bf (SMP)} and {\bf (HL)} and also bounded above. Let $s^*=\sup\mathcal{S} > 0$. Assuming by contradiction that $\Lambda > C_0 \Lambda_0$ ($\sigma$ is fixed) and applying the  above inequality and $u \geq (s^*)^{\hat{\alpha}} \varphi_0$ in $\Omega$, we derive

\begin{eqnarray*}
-\mathcal{L}(u-(s^*)^{\hat{\alpha}}\varphi_0) &\geq& \frac{1}{C_0} \Lambda\rho_0(x)S_\alpha(u) - (s^*)^{\hat{\alpha}} \Lambda_0 \rho_0(x) S_\alpha(\varphi_0) \\
&\geq& \left(\frac{1}{C_0} \Lambda - \Lambda_0\right) \rho_0(x)(s^*)^{\hat{\alpha}}S_\alpha(\varphi_0) > 0\ \ {\rm in}\ \Omega.
\end{eqnarray*}
Here it was used that $S_\alpha((s^*)^{\hat{\alpha}}\varphi_0) = (s^*)^{\hat{\alpha}} S_\alpha(\varphi_0)$. Finally, {\bf(SMP)} and {\bf(HL)} applied to the above inequality yield the contradiction $u > (s^*+\varepsilon)^{\hat{\alpha}}\varphi_0$ in $\Omega$ for $\varepsilon>0$ small enough. Therefore, we deduce that $\Lambda \leq C_0 \Lambda_0$, so that $\lambda \leq C_0 \theta_*$.
\end{proof}

Thanks to Lemmas \ref{existence_lema1} and \ref{existence_lema2}, for any $\sigma \in \R^{m-1}_+$, the number
\[
\lambda^*(\sigma)=\sup\mathcal{T}_\sigma
\]
is well defined and positive. Introduce the sets

\begin{eqnarray*}
&& \mathcal{A} = \{(\lambda,\lambda \sigma):\ 0 < \lambda < \lambda^*(\sigma),\ \sigma \in \R^{m-1}_+\},\\
&& \mathcal{B} = \{(\lambda,\lambda \sigma):\ \lambda > \lambda^*(\sigma),\ \sigma \in \R^{m-1}_+\},\\
&& \Lambda^* = \{(\lambda^*(\sigma),\lambda^*(\sigma) \sigma):\ \sigma \in \R^{m-1}_+\}.
\end{eqnarray*}
The set $\Lambda^*$ is called extremal set associated to the problem \eqref{1}.

The next result shows partially (i) of Theorem \ref{separation}.

\begin{lem}\label{existence_lema3} The set $\mathcal{A}_0$ contains $\mathcal{A}$. In particular, \eqref{1} admits a minimal positive strong solution $u_\Lambda$ for any $\Lambda \in {\cal A}$.
\end{lem}

\begin{proof}
Let a fixed $\sigma\in\mathbb{R}^{m-1}_+$. Given $\Lambda_0 = (\lambda_0,\lambda_0\sigma) \in \mathcal{A}_0$, it suffices to deduce that $(0, \lambda_0) \subset \mathcal{T}_\sigma$. Take $0 < \lambda < \lambda_0$ and set $\Lambda = (\lambda,\lambda\sigma)$. Let $u_{\Lambda_0} $ be a positive strong solution of \eqref{1} corresponding to $\Lambda_0$. Since $\Lambda \leq \Lambda_0$, we have

\[
-\mathcal{L}u_{\Lambda_0} = \Lambda_0 F(x,u_{\Lambda_0}) \geq \Lambda F(x,u_{\Lambda_0}) \text{ in }\Omega.
\]
In other words, the map $u_{\Lambda_0}$ is a supersolution of the problem \eqref{1}. Let $\{u_k\}_{k\geq1}$ be the sequence of maps constructed in the proof of Lemma \ref{existence_lema1}.  Proceeding in a similar manner to that proof, one easily concludes that $\{u_k\}_{k\geq1}$ is a pointwise nondecreasing sequence that satisfies $u_k \leq u_{\Lambda_0}$ in $\Omega$ for all $k \geq 1$. Then, the same argument applies as in Lemma \ref{existence_lema1} and so we deduce that the problem \eqref{1} has a positive strong solution. Therefore, $\Lambda \in \mathcal{A}_0$, so that $\lambda \in \mathcal{T}_\sigma$. This finishes the proof.
\end{proof}

We now ensure monotonicity and continuous of $u_\Lambda$ with respect to $\Lambda$.

\begin{lem}\label{monotonic-continuity}
	The map $\Lambda \in \mathcal{A} \mapsto u_\Lambda$ is nondecreasing and continuous with respect to the topology of $W^{2,n}(\Omega; \R^m) \cap W_0^{1,n}(\Omega; \R^m)$.
\end{lem}
\begin{proof}
	We first show that the map $\Lambda \in \mathcal{A} \mapsto u_\Lambda$ is nondecreasing. Choose any $\underline \Lambda, \overline \Lambda \in \mathcal{A}$ such that $\underline \Lambda \leq \overline \Lambda$. Notice that $u_{\overline \Lambda}$ is a positive strong supersolution of
	\begin{equation*}
	\left\{
	\begin{array}{rrll}
	-\mathcal{L} u &=& \underline \Lambda F(x, u) & {\rm in} \ \ \Omega,\\
	u&=&0 & {\rm on} \ \  \partial\Omega.
	\end{array}
	\right.
	\end{equation*}
Proceeding as in the proof of Lemma \ref{existence_lema3}, we readily get $u_{\underline \lambda} \leq u_{\overline \Lambda}$ in $\Omega$. Moreover, if $\underline \Lambda < \overline \Lambda$, then by (A) and (B) we have

\[
-\mathcal{L} u_{\underline \lambda} = \underline \Lambda F(x, u_{\underline \lambda}) \leq \underline \Lambda F(x, u_{\overline \Lambda}) < \overline \Lambda F(x, u_{\overline \Lambda}) = -\mathcal{L} u_{\overline \lambda}\ {\rm in}\ \Omega,
\]
so that $u_{\underline \lambda} < u_{\overline \lambda}$ in $\Omega$, by {\bf(SMP)}.
	
	We now prove that the map $\Lambda \in \mathcal{A} \mapsto u_\Lambda \in W^{2,n}(\Omega; \R^m) \cap W_0^{1,n}(\Omega; \R^m)$ is continuous at $\Lambda_0 \in \mathcal{A}$. Write $\Lambda_0 = (\lambda_0, \lambda_0 \sigma)$ with $\lambda_0 > 0$ and $\sigma \in \R_+^{m-1}$. Choose a fixed number $\overline \lambda$ such that $0 < \lambda_0 < \overline \lambda < \lambda^*(\sigma)$. For $0 < \lambda < \overline \lambda $, set $\Lambda = (\lambda, \lambda \sigma)$ and $\overline \Lambda = (\overline \lambda, \overline \lambda \sigma)$. Note that $\Lambda < \overline \Lambda$ for all $\Lambda$ close enough to $\Lambda_0$, so that $0 < u_\Lambda < u_{\overline \Lambda}$ in $\Omega$ by strict monotonicity. Since $F(x, u_{\overline \Lambda}) \in L^n(\Omega; \R^m)$, it follows from the $L^n$ Calderón-Zygmund theory that $u_\Lambda$ is uniformly bounded in $W^{2,n}(\Omega; \R^m)$ for $\Lambda$ around of $\Lambda_0$. But this implies that exists a positive strong solution $u_0 \in W^{2,n}(\Omega; \R^m) \cap W_0^{1,n}(\Omega; \R^m)$ of

\begin{equation*}
	\left\{
	\begin{array}{rrll}
	-\mathcal{L} u &=& \Lambda_0 F(x, u) & {\rm in} \ \ \Omega,\\
	u&=&0 & {\rm on} \ \  \partial\Omega.
	\end{array}
	\right.
	\end{equation*}
as the limit of $u_\Lambda$ as $\Lambda$ tends to $\Lambda_0$. In particular, $u_0 \geq u_{\Lambda_0}$ in $\Omega$.

On the other hand, for $\Lambda < \Lambda_0$ we have $u_\Lambda < u_{\Lambda_0}$ in $\Omega$. This leads readily to the reverse inequality $u_0 \leq u_{\Lambda_0}$ in $\Omega$ and so we end the proof.
\end{proof}

\begin{lem}\label{nonexistence}
The problem \eqref{1} admits no nonnegative strong solution for any $\Lambda \in \mathcal{B}$.
\end{lem}
\begin{proof}
As usual, the proof is carried out by contradiction. Assume that such a solution $u_0$ of \eqref{1} exists for some $\Lambda_0 \in \mathcal{B}$, which is clearly positive by (A) and (B). By definition of $\mathcal{B}$, we can write $\Lambda_0 = (\lambda_0, \lambda_0 \sigma)$ for some $\lambda_0 > \lambda^*(\sigma) > 0$ with $\sigma \in \R^{m-1}_+$. But the latter inequality contradicts the definition of $\lambda^*(\sigma)$.
\end{proof}

\section{Qualitative properties of $\Lambda^*$}

Consider the map $\Phi:\mathbb{R}^{m-1}_+\to\Lambda^*\subset\mathbb{R}^m_+$ defined by

\[
\Phi(\sigma)=(\lambda^*(\sigma),\nu^*(\sigma)):=(\lambda^*(\sigma),\lambda^*(\sigma)\sigma)
\]
which parameterizes the set $\Lambda^*$, where $\lambda^*(\sigma)$ was defined in the previous section.

This section is devoted to qualitative properties satisfied by $\Lambda^*$ by means of the map $\Phi$. In particular, we present the proof of Theorem \ref{properties} organized inti four propositions.

We first attach the part (I) of Theorem \ref{properties}.

\begin{prop}\label{continuity-components} The hypersurface $\Lambda^*$ is continuous.
\end{prop}

\begin{proof}
For the continuity of $\Lambda^*$, it suffices to show the continuity of $\Phi(\sigma)$ or, equivalently, $\lambda^*(\sigma)$. Assume by contradiction that $\lambda^*(\sigma)$ is discontinuous at some $\sigma \in \mathbb{R}^{m-1}_+$. Then, there exist a number $\varepsilon > 0$ and a sequence $\sigma_k$ converging to $\sigma$ such that, for any $k \geq 1$,
\[
|\lambda^*(\sigma_k)-\lambda^*(\sigma)| \geq \varepsilon.
\]
Module a subsequence, we can assume that

\[
\lambda^*(\sigma_k) \leq \lambda^*(\sigma) - \varepsilon
\]
or

\[
\lambda^*(\sigma_k) \geq \lambda^*(\sigma) + \varepsilon
\]
for $k$ sufficiently large. For our purposes, it is enough to consider only the first case, once the second one is carried out in an analogous way.

Take positive numbers $\underline{\lambda}$ and $\overline{\lambda}$ so that

\[
\lambda^*(\sigma_k)<\underline{\lambda}<\overline{\lambda}<\lambda^*(\sigma).
\]
Then,

\[
\lambda^*(\sigma_k)\sigma_k<\underline{\lambda}\sigma_k<\overline{\lambda}\sigma<\lambda^*(\sigma)\sigma
\]
for $k$ large enough.

Let $\underline{\Lambda}_k=(\underline{\lambda},\underline{\lambda}\sigma_k)$ and $\overline{\Lambda}=(\overline{\lambda},\overline{\lambda}\sigma)$. From the above inequality and Lemma \ref{existence_lema3}, it follows that \eqref{1} has a positive strong solution $\overline{u}$ corresponding to $\overline{\Lambda}$. Moreover, noting that $\overline{\Lambda} \geq \underline{\Lambda}_k$, we get

\[
\left\{\begin{array}{rlllr}
-\mathcal{L}\overline{u} &\geq& \underline{\Lambda}_kF(x,\overline{u}) &\text{ in }& \Omega,\cr
                     \overline{u}&=&0 &\text{ on }& \partial\Omega,\end{array}\right.
\]
which implies that $\overline{u}$ is a positive supersolution of \eqref{1} for $\Lambda = \underline{\Lambda}_k$. Hence, \eqref{1} possesses a positive strong solution for $\Lambda = \underline{\Lambda}_k$, so that $\underline{\lambda} \leq \lambda^*(\sigma_k)$ for $k$ large enough, contradicting the above reverse inequality. This concludes the proof.
\end{proof}

For a better understanding of the hypersurface $\Lambda^*$, we need to study the behavior of the components of $\Phi$. The next proposition proves the part (II) of Theorem \ref{properties}.
\begin{prop}\label{behavior-components-step} The function $\lambda^*(\sigma)$ is nonincreasing on $\sigma$.
\end{prop}
\begin{proof} Assume by contradiction that the  claim is false, that is, $\lambda^*(\sigma_1) < \lambda^*(\sigma_2)$ for some $\sigma_1 < \sigma_2$ in $\mathbb{R}^{m-1}_+$. Choose positive numbers $\underline{\lambda}$ and $\overline{\lambda}$  such that

\[
\lambda^*(\sigma_1) < \underline{\lambda} < \overline{\lambda}<\lambda^*(\sigma_2),
\]
then

\[
\lambda^*(\sigma_1)\sigma_1 < \underline{\lambda}\sigma_1 < \overline{\lambda}\sigma_2 < \lambda^*(\sigma_2)\sigma_2.
\]
Denoting $\underline{\Lambda} = (\underline{\lambda}, \underline{\lambda}\sigma_1)$ and $\overline{\Lambda} = (\overline{\lambda}, \overline{\lambda}\sigma_2)$, we have $\underline{\Lambda} < \overline{\Lambda}$. Then, applying exactly the same argument of the previous proposition to the these inequalities, we readily arrive at the contradiction $\underline{\lambda} \leq \lambda^*(\sigma_1)$. This ends the proof.
\end{proof}

We now prove the part (III) of Theorem \ref{properties}.

\begin{prop}
For any $0 < \sigma_1 \leq \sigma_2$, there exists $i\in\{1,\ldots,m\}$ such that $\nu^*_i(\sigma_1)\leq\nu^*_i(\sigma_2)$.
\end{prop}
\begin{proof}
Assume by contradiction that there exists $0 < \sigma_1 < \sigma_2$ such that $\nu^*(\sigma_2)<\nu^*(\sigma_1)$. This inequality and Proposition \ref{behavior-components-step} imply that $\lambda^*(\sigma_2) < \lambda^*(\sigma_1)$ and $\lambda^*(\sigma_2) \sigma_2 < \lambda^*(\sigma_1) \sigma_1$. Now, choose positive numbers $\underline{\lambda}$ and $\overline{\lambda}$ such that

\[
\lambda^*(\sigma_2) < \underline{\lambda} < \overline{\lambda}<\lambda^*(\sigma_1)
\]
and

\[
\lambda^*(\sigma_2)\sigma_2 < \underline{\lambda}\sigma_2 < \overline{\lambda}\sigma_1 < \lambda^*(\sigma_1)\sigma_1.
\]
Let $\underline{\Lambda} = (\underline{\lambda}, \underline{\lambda}\sigma_2)$ and $\overline{\Lambda} = (\overline{\lambda}, \overline{\lambda}\sigma_1)$. Thereby, one has $\underline{\Lambda} < \overline{\Lambda}$. Proceeding as in the proof of the Proposition \ref{continuity-components}, we obtain the contradiction $\underline{\lambda} \leq \lambda^*(\sigma_2)$. Thus, we finish the proof.
\end{proof}

Finally, we derive the asymptotic behavior of the hypersurface $\Lambda^*$ stated in the part (IV) of Theorem \ref{properties}.
\begin{prop}
The limit $\lambda^*(\sigma) \to 0$ as $\sigma_i \to \infty$ occurs for each $i = 1, \ldots, m-1$.
\end{prop}
\begin{proof}
Let a fixed $i = 1, \ldots, m-1$. It was proved in Lemma \ref{existence_lema2} that $\lambda^*(\sigma) \leq C_0 \mu_1(\sigma)$ for all $\sigma \in \R^{m-1}_+$, where $C_0 > 0$ is a constant independent of $\sigma$. But, thanks to the characterization of $\mu_1(\sigma)$ provided in Corollary \ref{Autovalor}, we know that $\mu_1(\sigma) \to 0$ as $\sigma_i \to \infty$. It then follows the desired conclusion.
\end{proof}

\section{Weak solutions on $\Lambda^*$}

The proof of Theorem \ref{fraca} requires $L^1$ {\it a priori} estimates for strong solutions of \eqref{1} when $\Lambda$ belongs to a neighborhood of the hypersurface $\Lambda^*$.

Let $\delta(x)={\rm dist}(x,\partial\Omega)$. We recall that the $\delta$-weighted $L^1$ space is given by

\[
L^1(\Omega,\delta(x) dx) = L^1(\Omega,\delta) := \left\{h: \Omega \rightarrow \R:\ \int_\Omega |h(x)| \delta(x) dx<\infty\right\}
\]
and endowed with the norm $ \|h\|_{L^1(\Omega,\delta)}=\int_\Omega|h(x)|\delta(x) dx$.

The next two lemmas are useful tools in our proof. The first of them is a straightforward adaptation of the proof of Lemma 1 of \cite{BCMR} since the essential ingredient is the maximum principle assumed for ${\cal L}$. The second one is a consequence of global estimates for Green functions associated to elliptic operators on $C^{1,1}$ domains established in \cite{Ancona1982}, \cite{HueberSieveking1982} and \cite{Zhao1986}.

\begin{lem}\label{uniqueness-weak-solutions-lemma} Let ${\cal L}$ be a uniformly elliptic operator such that $a_{kl} \in C^2(\overline{\Omega})$, $b_j \in C^1(\overline{\Omega})$ for every $k,l, j$, $c \in L^\infty(\Omega)$ and $\mu_1(-{\cal L}, \Omega) > 0$. Then, given $h \in L^1(\Omega,\delta)$, the problem

\begin{gather}\tag{5.1}\label{BrCa-equation}
\left\{
\begin{array}{rlllr}
-\mathcal{L} v &=& \ h(x) & {\rm in} & \Omega, \\
v&=&0 & {\rm on} & \partial\Omega
\end{array}
\right.
\end{gather}
admits a unique weak solution $v \in L^1(\Omega)$. Moreover, there exists a constant $C_1 > 0$ independent of $h$ such that

\[
\|v\|_{L^1(\Omega)} \leq C_1 \|h\|_{L^1(\Omega,\delta)}.
\]
Besides, if $h \geq 0$ almost everywhere in $\Omega$, then $v\geq0$ almost everywhere in $\Omega$.
\end{lem}

\begin{lem}\label{HL-estimate}
Let ${\cal L}$ be a uniformly elliptic operator such that $a_{kl} \in C^2(\overline{\Omega})$, $b_j \in C^1(\overline{\Omega})$ for every $k,l, j$, $c \in L^\infty(\Omega)$ and $\mu_1(-{\cal L}, \Omega) > 0$. Let $h \in L^\infty(\Omega)$ such that $h\geq0$ almost everywhere in $\Omega$. Then, the strong solution $v \in W^{2,n}(\Omega) \cap W_0^{1,n}(\Omega)$ of \eqref{BrCa-equation} satisfies

\[
v(x) \geq C_2 \delta(x) \|h\|_{L^1(\Omega,\delta)}
\]
for $x \in \Omega$ almost everywhere, where $C_2$ is a positive constant independent of $h$.
\end{lem}
\begin{proof}
Let $G_{\cal L}$ and $G_{\Delta}$ be Green's functions associated to the operators $-\mathcal{L}$ and $-\Delta$ with zero boundary condition. One knows, independently from \cite{Ancona1982} and \cite{HueberSieveking1982}, that there exists a constant $C > 0$, depending on $\Omega$ and coefficients of $\mathcal{L}$, such that for any $(x,y) \in \Omega \times \Omega$ with $x \neq y$ and $x, y \neq 0$,

\[
\frac1C G_{\Delta}(x,y) \leq G_{\cal L} (x,y) \leq C G_{\Delta}(x,y).
\]
On the other hand, as proved in \cite{Zhao1986}, there exists a constant $\tilde{C} > 0$, depending only on $\Omega$, such that for any $(x,y) \in \Omega \times \Omega$ with $x \neq y$ and $x, y \neq 0$,

\[
G_\Delta(x,y)\geq\tilde{C}\delta(x)\delta(y).
\]
For these estimates, we get a constant $C_2 > 0$, depending on $\Omega$ and coefficients of $\mathcal{L}$, such that for any $(x,y) \in \Omega \times \Omega$ with $x \neq y$ and $x, y \neq 0$,

\[
G_{\cal L}(x,y) \geq C_2\delta(x)\delta(y).
\]
Thus, the strong solution $v \in W^{2,n}(\Omega) \cap W_0^{1,n}(\Omega)$ of \eqref{BrCa-equation} satisfies

\[
v(x)=\int_\Omega G_{\cal L}(x,y) h(y) dy \geq C_2 \delta(x)\left(\int_\Omega h(y)\delta(y)\ dy\right).
\]
This ends the proof.
\end{proof}

Before stating the next lemma, we quote a direct consequence of the assumptions (A) and (C).

Let $\kappa$, $\rho$, $\alpha$ and $M$ be as in (C). Thanks to the positivity of $F$ and $\rho$ and the compactness of $\{t \in \R^n: |t| \leq M\}$, there exists a constant $B > 0$ depending on $\kappa$, so that

\begin{gather}\tag{5.2} \label{lower}
F(x,t) \geq \kappa \rho(x) S_\alpha(t)-B \rho(x)
\end{gather}
for $x \in \Omega$ almost everywhere and $t \in \R^m_+$.

In the following result we establish an {\it a priori} estimate to strong solutions of \eqref{1} by using Lemmas \ref{uniqueness-weak-solutions-lemma} and \ref{HL-estimate}.

\begin{lem}\label{boundedness-of-u}
Let a fixed $\sigma \in \R_+^{m-1}$ and ${\cal L}_i$ be a uniformly elliptic operator such that $a^i_{kl} \in C^2(\overline{\Omega})$, $b^i_j \in C^1(\overline{\Omega})$ and $c^i \in L^\infty(\Omega)$ for every $i, k, l, j$. Let $F:\Omega\times \R^m \to \R^m$ be a map such that $F(x, \cdot): \R^m \to \R^m$ is continuous for $x \in \Omega$ almost everywhere and $F(\cdot, t): \Omega \to \R^m$ belongs to $L^n(\Omega; \R^m)$ for every $t \in \R^m$. Assume also that $F$ satisfies (A), (B) and (C). Then, there exists a constant $C_3 > 0$ such that for any $\frac{1}{2}\lambda^*(\sigma) < \lambda < \lambda^*(\sigma)$,

\[
\|u\|_{L^1(\Omega;\mathbb{R}^m)} \leq C_3
\]
for all positive strong solutions $u$ of \eqref{1} corresponding to $\Lambda=(\lambda,\lambda\sigma)$.
\end{lem}
\begin{proof}
Assume by contradiction that there is no such a bound for strong solutions. In other words, there exist $\lambda_k \in (\frac{1}{2}\lambda^*(\sigma) , \lambda^*(\sigma))$ and a positive strong solution $u_k$ of \eqref{1} associated to $\Lambda_k=(\lambda_k,\lambda_k\sigma)$ such that $\|u_k\|_{L^1(\Omega;\mathbb{R}^m)} \to \infty$ as $k \rightarrow \infty$.

Let $s_1 = 1$ and $s_{i+1} = \alpha_i s_i$ for $i = 1, \ldots, m$. Note that $s_{m+1} = s_1$, since $\Pi \alpha =1$. Observing that all $s_i$ are positive, there is a leading index $l \in \{1, \ldots, m\}$ in the sense that, module a subsequence,

\[
\|u_{l+1}^k\|_{L^1(\Omega)} \to \infty \text{ and }\|u_{l+1}^k\|^{s_{l+1}}_{L^1(\Omega)} \geq \|u_i^k\|^{s_i}_{L^1(\Omega)} \text{   for all } i = 1, \ldots, m.
\]
In particular, for $i = l$, we have

\begin{gather}\tag{5.3}\label{5.3-eq2}
\|u_{l+1}^k\|_{L^1(\Omega)} \to \infty \text{ and }\|u_{l+1}^k\|^{\alpha_l}_{L^1(\Omega)} \geq \|u_l^k\|_{L^1(\Omega)}.
\end{gather}

Let $\zeta^*_l \in W^{2,n}(\Omega) \cap W_0^{1,n}(\Omega)$ be the strong solution of

\[
\left\{\begin{array}{rlllr}
-\mathcal{L}_{l}^*\zeta^*_l&=&1 & \text{ in }& \Omega,\cr
\zeta^*_l&=&0 &\text{ on }& \partial\Omega.
\end{array}\right.
\]
Since $\lambda_k > \frac{1}{2}\lambda^*(\sigma)$, we can choose $\kappa$ in \eqref{lower} such that

\[
\kappa(C_1C_2)^{\alpha_{l}}\lambda_k\sigma_{l-1}\int_\Omega \rho_{l}(x)\delta(x)^{\alpha_{l} }\zeta^*_l(x) dx \geq 2
\]
for every $k \geq 1$, where $C_1$ and $C_2$ are the positive constants given respectively in Lemmas \ref{uniqueness-weak-solutions-lemma} and \ref{HL-estimate}.

Let $e_k=\| u_{l+1}^k\|_{L^1(\Omega)}$. Thanks to \eqref{lower} and \eqref{5.3-eq2}, we get

\begin{eqnarray*}
1 &\geq& e_k^{-\alpha_{l}} \int_\Omega u_{l}^k\ dx = e_k^{-\alpha_{l}} \int_\Omega \left(-\mathcal{L}_{l} u_{l}^k\right)\zeta^*_l dx = e_k^{-\alpha_{l}} \lambda_k \sigma_{l-1} \int_\Omega f_{l}(x,u_k) \zeta^*_l dx\cr
&\geq& e_k^{-\alpha_{l}} \kappa\lambda_k \sigma_{l-1} \int_\Omega \rho_{l}(x)\left(u^k_{l+1}\right)^{\alpha_{l}} \zeta^*_l dx-B e_k^{-\alpha_{l}}\lambda_k\sigma_{l-1}\int_\Omega \rho_{l}(x)\zeta^*_l dx \cr .
\end{eqnarray*}
The first term on the right-hand side is then estimated as follows. By Lemmas \ref{uniqueness-weak-solutions-lemma} and \ref{HL-estimate}, we have

\[
u_{l+1}^k(x) \geq C_2 \delta(x)\int_\Omega \lambda_k\sigma_{l} f_{l+1}(x,u_k)\delta(x) dx \geq C_1 C_2\delta(x) \|u_{l+1}^k\|_{L^1(\Omega)},
\]
so that

\[
e_k^{-\alpha_{l}}\int_\Omega \rho_{l}(x)\left(u^k_{l+1}\right)^{\alpha_{l}}\zeta^*_l dx \geq \left(C_1C_2\right)^{\alpha_{l}}\int_\Omega \rho_{l}(x)\delta(x)^{\alpha_{l}}\zeta^*_l dx.
\]
Finally, letting the limit $k \to \infty$ in the inequalities

\begin{eqnarray*}
1 &\geq& \kappa\lambda_k e_k^{-\alpha_{l}}\sigma_{l-1}\int_\Omega \rho_{l}(x)\left(u^k_{l+1}\right)^{\alpha_{l}}\zeta^*_l dx-B e_k^{-\alpha_{l}}\lambda_k \sigma_{l-1}\int_\Omega \rho_{l}(x)\zeta^*_l dx \cr
&\geq& \kappa\left(C_1C_2\right)^{\alpha_{l}}\lambda_k\sigma_{l-1}\int_\Omega \rho_{l}(x)\delta(x)^{\alpha_{l} }\zeta^*_l(x) dx-B e_k^{-\alpha_{l}}\lambda_k\sigma_{l-1}\int_\Omega \rho_{l}(x)\zeta^*_l dx \cr
&\geq& 2-O\left( e_k^{-\alpha_{l}}\right),
\end{eqnarray*}
we arrive at a contradiction. This concludes the proof.
\end{proof}

\begin{proof}[Proof of Theorem \ref{fraca}]
Let $\Lambda = (\lambda^*(\sigma),\lambda^*(\sigma)\sigma) \in \Lambda^*$ for some fixed $\sigma \in \R_+^{m-1}$. The proof that exists a minimal nonnegative weak solution for \eqref{1} associated to $\Lambda$ is done by approximation. Take $\lambda_k$ satisfying $0 < \lambda_k < \lambda^*(\sigma)$ and $\lambda_k \uparrow \lambda^*(\sigma)$ as $k \to \infty$. For each $k$, let $u_k$ be the minimal positive strong solution corresponding to $\Lambda_k := (\lambda_k, \lambda_k \sigma)$. By $(I)$ of Theorem \ref{separation} and the condition (B), it follows that $F(x,u_k)$ is nondecreasing on $k$. Take now the strong solution $\zeta^* \in W^{2,n}(\Omega;\mathbb{R}^m) \cap W_0^{1,n}(\Omega;\mathbb{R}^m)$ of the problem

\[
\left\{\begin{array}{rlllr}
-\mathcal{L}^*\zeta^* &=& 1& \text{ in }& \Omega,\cr
\zeta^* &=& 0 &\text{ on }& \partial\Omega
\end{array}\right.
\]
as a test function in \eqref{1}, so that

\[
\Lambda_k \int_\Omega F(x,u_k)\zeta^* dx = \int_\Omega u_k dx.
\]
Invoking Lemma \ref{boundedness-of-u}, we conclude that $u_k$ and $F(\cdot,u_k) \delta(\cdot)$ are uniformly bounded in $L^1(\Omega;\mathbb{R}^m)$. Therefore, by the monotone convergence theorem, $(u_k)$ and $(F(\cdot,u_k) \delta(\cdot))$ converge respectively to $u^*$ and $F(\cdot, u^*) \delta(\cdot)$ in $L^1(\Omega;\mathbb{R}^m)$. So, letting $k \rightarrow \infty$ in the equality

\[
\int_\Omega u_k(-\mathcal{L}^* \zeta) dx = \Lambda_k \int_\Omega F(x,u_k) \zeta dx,
\]
where $\zeta \in W^{2,n}(\Omega; \R^m) \cap W_0^{1,n}(\Omega; \R^m)$ satisfies ${\cal L}^* \zeta \in L^\infty(\Omega; \R^m)$, we deduce that $u^*$ is a nonnegative weak solution of \eqref{1} associated to $\Lambda$.

Finally, one easily shows that $u^*$ is minimal. Indeed, let $v$ be another nonnegative weak solution of \eqref{1}. By the construction of $u_k$ in the proof of Lemma \ref{existence_lema1} and, by Lemma \ref{HL-estimate}, we conclude that $u_k\leq v$ almost everywhere in $\Omega$, which clearly lead to the minimality of $u^*$.
\end{proof}

\section{Stability of minimal solutions}

We now dedicate special attention to the proof of Theorem \ref{stability}. We recall that the minimal positive strong solution $u = u_\Lambda$ of \eqref{1} for $\Lambda \in {\cal A}$ is said to be stable if the eigenvalue problem \eqref{4} has a principal eigenvalue $\eta_1 \geq 0$ and asymptotically stable if $\eta_1 > 0$.

\begin{proof}[Proof of Theorem \ref{stability}]

As mentioned in the introduction, the existence of a principal eigenvalue $\eta_1 = \eta_1(-{\cal L} - A(x,u_\Lambda))$ for \eqref{4} is ensured by \eqref{5} and assumptions (B) and (D).

We first assert that $u_\Lambda$ is stable. Assume by contradiction that $\eta_1 < 0$. Let $\varphi \in W^{2,n}(\Omega; \R^m) \cap W_0^{1,n}(\Omega; \R^m)$ be a positive eigenfunction associated to $\eta_1$. Set $\overline{u}_\varepsilon = u_\Lambda - \varepsilon \varphi$ for $\varepsilon > 0$. Clearly, we have $\overline{u}_\varepsilon = 0$ on $\partial \Omega$ and, by Hopf's Lemma, $\overline{u}_\varepsilon > 0$ in $\Omega$ for $\varepsilon > 0$ small enough.

Using the hypothesis that $F(x, \cdot): \R^m \to \R^m$ is of $C^1$ class uniformly on $x \in \Omega$ and $t$ in compacts of $\R^m$, we get for any $\varepsilon > 0$ small enough,

\begin{eqnarray*}
-{\cal L} \overline{u}_\varepsilon &=& -{\cal L} u_\Lambda + \varepsilon{\cal L} \varphi\\
&=& \Lambda F(x, u_\Lambda) + \varepsilon \left( - \Lambda (A(x, u_\Lambda) \varphi) - \eta_1 \varphi\right)\\
&=& \Lambda F(x, \overline{u}_\varepsilon) + r(x, \varepsilon \varphi) - \eta_1 \varepsilon \varphi \\
&>& \Lambda F(x, \overline{u}_\varepsilon) \ \ {\rm in}\ \Omega
\end{eqnarray*}
For the latter inequality, it was used that $\eta_1 < 0$, $\varphi > 0$ in $\Omega$, $\varphi = 0$ on $\partial \Omega$, $r(x, \varepsilon \varphi) = o(\varepsilon \sum_j \varphi_j)$ and Hopf's Lemma. Thus, $\overline{u}_\varepsilon$ is a positive supersolution of eqref{1} as well as the null map is a subsolution by (A). Mimicking the construction presented in the proof of Lemma \ref{existence_lema1}, we obtain a positive strong solution $u$ for \eqref{1} satisfying $u \leq \overline{u}_\varepsilon$ in $\Omega$ for $\varepsilon > 0$ small enough. But this contradicts the fact that $\overline{u}_\varepsilon < u_\Lambda$ in $\Omega$ and $u_\Lambda$ is a minimal positive strong solution of \eqref{1}. Hence, we conclude that $\eta_1 \geq 0$.

We now focus on asymptotic stability. Let $\Lambda = (\lambda_1, \ldots, \lambda_m)$ and $u_\Lambda = (u_1, \ldots, u_m)$. Denote by $B(x)$ the matrix whose elements are given by $B_{ij}(x) = \lambda_i \frac{\partial f_i}{\partial u_j}(x, u_\Lambda(x))$ and by $B(x)^T$ its transposed. We know that $\eta_1 = \eta_1(-{\cal L} - B(x))$.

Suppose by contradiction that $\eta_1 = 0$. Assume that $a^i_{kl} \in C^2(\overline{\Omega})$ and $b^i_j \in C^1(\overline{\Omega})$ for every $i, j, k, l$. In this case, since $\eta_1 = \eta_1^* = \eta_1(-{\cal L}^* - B(x)^T)$, there exists a positive eigenfunction $\varphi^* = (\varphi^*_1, \ldots, \varphi^*_m) \in W^{2,n}(\Omega; \R^m) \cap W_0^{1,n}(\Omega; \R^m)$ associated to $\eta_1^*$, that is,

\begin{gather}\tag{6.1} \label{6}
\left\{
\begin{array}{rlllr}
-{\cal L}^* \varphi^* &=& B(x)^T \varphi^* & {\rm in} & \Omega, \\
\varphi^* &=& 0 & {\rm on}& \partial \Omega.
\end{array}\right.
\end{gather}
Multiplying the $i$th equation of \eqref{6} by $u_i$, integrating by parts, using the ith equation satisfied by $u_\Lambda$ in \eqref{1} and lastly summarizing on $i$, we get

\[
\sum_{i = 1}^m \int_\Omega \varphi^*_i \lambda_i f_i(x, u_\Lambda(x)) dx = \sum_{i = 1}^m \int_\Omega \sum_{j = 1}^m \lambda_j \frac{\partial f_j}{\partial u_i}(x, u_\Lambda(x)) \varphi^*_j u_i.
\]
Renaming the indexes $i$ and $j$ on the right-hand side, the equality can be rewritten as

\[
\sum_{i = 1}^m \int_\Omega \left( \lambda_i f_i(x, u_\Lambda(x)) - \sum_{j = 1}^m \lambda_i \frac{\partial f_i}{\partial u_j}(x, u_\Lambda(x)) u_j \right) \varphi^*_i dx = 0,
\]
or succinctly,

\[
\int_\Omega \left( \Lambda F(x, u_\Lambda) - \Lambda (A(x, u_\Lambda) u_\Lambda\right) \varphi^* dx = 0.
\]
Arguing in a similar way with $\Upsilon \in {\cal A}$, we also obtain

\[
\int_\Omega \left( \Upsilon F(x, u_\Upsilon) - \Lambda (A(x, u_\Lambda) u_\Upsilon\right) \varphi^* dx = 0.
\]
Subtracting both above inequalities, we easily discover that

\[
\int_\Omega \Lambda \left( F(x, u_\Upsilon) - F(x, u_\Lambda) - A(x, u_\Lambda) (u_\Upsilon - u_\Lambda) \right) \varphi^* dx = \int_\Omega (\Lambda - \Upsilon) F(x, u_\Upsilon) \varphi^* dx.
\]
Choosing $\Upsilon > \Lambda$ and using that $F$ is convex, we arrive at a contradiction, once the above right-hand side is negative, while the left-hand side is nonnegative. This completes the proof.
\end{proof}

\section{Regularity of extremal solutions for $n=2$ and $n=3$}

In this section we consider the problem

\begin{gather}\tag{7.1}\label{gradient-system}
\left\{
\begin{array}{rlllr}
-{\Delta} u &=& \Lambda \nabla f(u) & {\rm in} & \Omega, \\
u&=&0 & {\rm on} & \partial \Omega,
\end{array}\right.
\end{gather}
where $\Delta u = (\Delta u_1,\ldots,\Delta u_m)$, $\Lambda = (\lambda_1,\ldots, \lambda_m) \in \R^m_+$, $f:\R^m_+ \rightarrow \R$ is a positive $C^2$ function and $\Lambda \nabla f(u)=(\lambda_1f_{u_1}(u),\ldots,\lambda_mf_{u_m}(u))$ with $m \geq 1$.

Assume that the potential field $F = \nabla f$ satisfies the conditions (A), (B) and (C) and, moreover, $\Hessian f(t)>0$ for all $t \in \R^m_+$. The latter implies that the Jacobian matrix of $F$ satisfies (D).

Under these conditions, consider the parameter $\Lambda = (\lambda^*(\sigma), \lambda^*(\sigma) \sigma) \in \Lambda^*$ for a fixed $\sigma \in \R^{m-1}_+$. Let $(\lambda_k)$ be a sequence converging to $\lambda^*(\sigma)$ so that $0 < \lambda_k < \lambda^*(\sigma)$ for every $k \geq 1$. Then, each minimal positive strong solution $u_{\Lambda_k}$ for $\Lambda_k = (\lambda_k, \lambda_k \sigma)$ is stable and its $L^1$-limit

\[
\lim_{k \rightarrow \infty} u_{\Lambda_k} = u_\Lambda^*
\]
is an extremal solution associated to $\Lambda$.

In this section we will prove that $u_\Lambda^*$ is bounded when $\Omega$ is convex and $n=2, 3$. The key point here is an estimate for the solutions $u_{\Lambda_k}$ in the space $W^{1,4}$ in a neighborhood of $\partial\Omega$.

\begin{prop}\label{prop-boundary-estimative}
Let $f\in C^{2}(\mathbb{R}^m)$ be a positive function satisfying $\Hessian f(t)>0$ for all $t \in\R^m_+$. Assume that the potential field $F = \nabla f$ satisfies (A), (B) and (C). Let also $\Lambda = (\lambda_1,\ldots, \lambda_m) \in \mathcal{A}$ and $u_{\Lambda}=(u_1,\ldots,u_m)$ be the stable minimal positive strong solution of (\ref{gradient-system}). If $n=2$ or $n=3$, then for any $t>0$,

\begin{gather}
\tag{7.2}\label{boudary-estimate}
\begin{array}{lcl}
\|u_i\|_{L^{\infty}(\Omega)} & \leq & t + \displaystyle\dfrac{C(n)}{t}|\Omega|^{\frac{4-n}{2n}}\sqrt{\lambda_i}\left(\sum_k\frac{1}{\lambda_k}\int_{\{u_k<t\}}|\nabla u_k|^4dx\right)^{1/2}\nonumber\cr
& & + C(n)|\Omega|^{\frac{4-n}{2n}}\displaystyle\sqrt{\lambda_i}\left(\sum_{k<l}\int_{\Omega\setminus\{u_k \geq t,u_l\geq t\}}f_{u_ku_l}(u)|\nabla u_k||\nabla u_l|dx\right)^{1/2},\end{array}
\end{gather}
where $\{u_k<t\}=\{x\in\Omega \mid u_k(x)<t\}$.

\end{prop}

Two essential ingredients in the proof of Proposition \ref{prop-boundary-estimative} are taken of the work \cite{Fazly2013} about potential systems by Fazly and Ghoussoub. The first of them is the inequality

\begin{gather}\tag{7.3}\label{stability-inequality}
\sum_{i,j}\int_{\Omega}f_{u_iu_j}(u_\Lambda)\psi_i\psi_jdx\leq\sum_i\frac{1}{\lambda_i}\int_{\Omega}|\nabla\psi_i|^2dx,
\end{gather}
which follows directly from the stability definition applied to (10) in \cite{Fazly2013}. Already the second one is a consequence of (32) in \cite{Fazly2013}, namely,
\begin{gather}\tag{7.4}\label{7.4}
\begin{array}{lcl}
\displaystyle\sum_i\frac{1}{\lambda_i}\int_{\Omega\cap\{|\nabla u_i|\neq0\}}\left(|\nabla_T|\nabla u_i||^2+|A_i|^2|\nabla u_i|^2\right)\eta_i^2dx & \leq & \displaystyle\sum_i\frac{1}{\lambda_i}\int_{\Omega}|\nabla u_i|^2|\nabla\eta_i|^2dx \\
& & + \displaystyle\sum_{i<j}\int_{\Omega}f_{u_iu_j}(u)|\nabla u_i||\nabla u_j|(\eta_i-\eta_j)^2dx
\end{array}
\end{gather}
for any Lipschitz function $\eta_i:\overline{\Omega}\rightarrow\R$ satisfying $\eta_i\mid_{\partial\Omega}\equiv0$, where $\nabla_T$ denotes the tangential gradient along a level set of $u_i$ and

$$
|A_i|^2 = \sum_{l=1}^{n-1}\kappa_{i,l}^2
$$
with $\kappa_{i,l}$ being principal curvatures of the level sets of $u_i$ passing through $x\in\Omega\cap\{|\nabla u_i|\neq0\}$.

\begin{proof}[Proof of Proposition \ref{prop-boundary-estimative}]
Given $\Lambda \in \mathcal{A}$, let $u_{\Lambda}=(u_1,\ldots,u_m)$ be the stable minimal positive strong solution of (\ref{gradient-system}). For each $i = 1, \ldots, m$, set $T_{u_i} := \|u_i\|_{L^{\infty}(\Omega)}$ and $\Gamma_s^{u_i} := u_i^{-1}(s)$ for $s \in [0,T_{u_i}]$. By Sard's theorem, $s\in(0,T_{u_i})$ is a regular value of $u_i$ almost everywhere. In \eqref{7.4}, choose

$$
\eta_i(x) = \varphi(u_i(x)),
$$
where $\varphi$ is a Lipschitz function in $[0,\infty)$ such that $\varphi(0)=0$. By the coarea formula, we have

\[
\frac{1}{\lambda_i}\int_{\Omega}|\nabla u_i|^2|\nabla\eta_i|^2dx = \frac{1}{\lambda_i}\int_{\Omega}|\nabla u_i|^4\varphi'(u_i)^2dx  = \frac{1}{\lambda_i}\int_0^{T_{u_i}}\left(\int_{\Gamma_s^{u_i}}|\nabla u_i|^3dV_s\right)\varphi'(s)^2ds.
\]
This equality and \eqref{stability-inequality} together imply

$$
\sum_i\frac{1}{\lambda_i}\int_0^{T_{u_i}}\left(\int_{\Gamma_s^{u_i}}|\nabla u_i|^3dV_s\right)\varphi'(s)^2ds  +  \sum_{i<j}\int_{\Omega}f_{u_iu_j}(u)|\nabla u_i||\nabla u_j|(\varphi(u_i)-\varphi(u_j))^2dx
$$
\begin{eqnarray*}
&\geq & \displaystyle{\sum_i\dfrac{1}{\lambda_i}\int_{\Omega\cap\{|\nabla u_i|\neq0\}}\left(|\nabla_T|\nabla u_i||^2+|A_i|^2|\nabla u_i|^2\right)\varphi(u_i)^2dx}\\
& \geq & \sum_i\dfrac{1}{\lambda_i}\int_0^{T_{u_i}}\left(\int_{\Gamma_s^{u_i}}|A_i|^2|\nabla u_i|dV_s\right)\varphi(s)^2ds.
\end{eqnarray*}
More specifically, taking

$$\varphi(s)=\left\{\begin{array}{ll}
s/t &\mbox{if}\quad 0\leq s < t,\\
1 &\mbox{if}\quad t\leq s
\end{array}\right.
$$
in the above inequality, we get

\begin{eqnarray*}
\dfrac{1}{\lambda_i}\int_{t}^{T_{u_i}}\int_{\Gamma_s^{u_i}}|A_i|^2|\nabla u_i|dV_sds & \leq & \sum_k\dfrac{1}{\lambda_k}\int_0^{T_{u_k}}\left(\int_{\Gamma_s^{u_k}}|A_k|^2|\nabla u_k|dV_s\right)\varphi(s)^2ds \\
& \leq & \sum_k\frac{1}{\lambda_kt^2}\int_0^{t}\left(\int_{\Gamma_s^{u_k}}|\nabla u_k|^3dV_s\right)ds\\
& & + \sum_{k<l}\int_{\Omega}f_{u_ku_l}(u)|\nabla u_k||\nabla u_l|(\varphi(u_k)-\varphi(u_l))^2dx.
\end{eqnarray*}
In conclusion, the expression of $\varphi$ and the coarea formula provide

\begin{gather}\tag{7.5}\label{estimate1}
\begin{array}{lcl}
\dfrac{1}{\lambda_i}\displaystyle\int_{t}^{T_{u_i}}\displaystyle\int_{\Gamma_s^{u_i}}|A_i|^2|\nabla u_i|dV_sds & \leq & \displaystyle\sum_k\frac{1}{\lambda_kt^2}\int_{\{u_k<t\}}|\nabla u_k|^4dx\nonumber\cr
& & + \displaystyle\sum_{k<l}\int_{\Omega\setminus\{u_k \geq t,u_l\geq t\}}f_{u_ku_l}(u)|\nabla u_k||\nabla u_l|dx.
\end{array}
\end{gather}
Denote by $|\Gamma_s^{u_i}|$ the volume of $\Gamma_s^{u_i}$ and $H_i$ the mean curvature function of $\Gamma_s^{u_i}$. For any $n \geq 2$, the geometric inequality

\begin{gather}\tag{7.6}\label{desg1}
|\Gamma_s^{u_i}|^{\frac{n-2}{n-1}} \leq C(n)\int_{\Gamma_s^{u_i}}|H_i|dV_s
\end{gather}
holds for almost every $s$. In dimension $n=2$, the set $\Gamma_s^{u_i}$ is a regular curve for almost every $s$ and \eqref{desg1} follows from the theory of plane curves. For $n\geq3$, the inequality \eqref{desg1} is a consequence of Theorem 2.1 of \cite{Michael1973} by Michael and Simon (see also Mantegazza \cite{Mantegazza2002}, Proposition 5.2).\\ On the other hand, the isoperimetric inequality ensures that

\begin{gather}\tag{7.7}\label{desg2}
V_i(s):=|\{u_i>s\}| \leq C(n)|\Gamma_s^{u_i}|^{\frac{n}{n-1}}.
\end{gather}
Joining (\ref{desg1}) and (\ref{desg2}) and applying Hölder's inequality, we derive

\begin{eqnarray*}
 V_i(s)^{\frac{n-2}{n}} & \leq & C(n)\int_{\Gamma_s^{u_i}}|H_i|dV_s\\
 & \leq & C(n)\left\{\int_{\Gamma_s^{u_i}}|A_i|^2|\nabla u_i|dV_s\right\}^{1/2}\left\{\int_{\Gamma_s^{u_i}}\dfrac{dV_s}{|\nabla u_i|}\right\}^{1/2}.
\end{eqnarray*}
Here it was used that $|H_i| \leq |A_i|$. Thus, we have

\begin{eqnarray*}
\dfrac{1}{\sqrt{\lambda_i}}(T_{u_i} - t) & = & \dfrac{1}{\sqrt{\lambda_i}}\int_{t}^{T_{u_i}}ds\\
& \leq & \dfrac{1}{\sqrt{\lambda_i}}\int_{t}^{T_{u_i}}C(n)\left\{\int_{\Gamma_s^{u_i}}|A_i|^2|\nabla u_i|dV_s\right\}^{1/2}\left\{ V_i(s)^{\frac{2(2-n)}{n}}\int_{\Gamma_s^{u_i}}\dfrac{dV_s}{|\nabla u_i|}\right\}^{1/2}ds\\
& \leq & C(n)\left\{\dfrac{1}{\lambda_i}\int_{t}^{T_{u_i}}\int_{\Gamma_s^{u_i}}|A_i|^2|\nabla u_i|dV_sds\right\}^{1/2} \left\{\int_{t}^{T_{u_i}}V_i(s)^{\frac{2(2-n)}{n}}\int_{\Gamma_s^{u_i}}\dfrac{dV_s}{|\nabla u_i|}ds\right\}^{1/2}.
\end{eqnarray*}
Thanks to the estimate (\ref{estimate1}), we arrive at

\begin{eqnarray*}
\dfrac{1}{\sqrt{\lambda_i}}(T_{u_i} - t) & \leq & C(n)\left\{\sum_k\frac{1}{\lambda_kt^2}\int_{\{u_k<t\}}|\nabla u_k|^4dx +  \sum_{k<l}\int_{\{u_k \geq t,u_l\geq t\}}f_{u_ku_l}(u)|\nabla u_k||\nabla u_l|dx\right\}^{1/2}\\
& &\hspace{5.0cm} \times \left\{\int_{t}^{T_{u_i}}V_i(s)^{\frac{2(2-n)}{n}}\int_{\Gamma_s^{u_i}}\dfrac{dV_s}{|\nabla u_i|}ds\right\}^{1/2}.
\end{eqnarray*}
Since the function $V_i(t)^{\frac{4-n}{n}}$ is nonincreasing for $n \leq 3$, again applying the coarea formula, we get

$$
-V_i'(s)=\int_{\Gamma_s^{u_i}}\dfrac{dV_s}{|\nabla u_i|}
$$
for $s \in (0,T_{u_i})$ almost everywhere. We also have

$$
|\Omega|^{\frac{4-n}{n}}\geq V_i(t)^{\frac{4-n}{n}}=\left.V_i(s)^{\frac{4-n}{n}}\right|_{s=T_{u_i}}^{s=t} \geq \dfrac{4-n}{n}\int_{t}^{T_{u_i}}V_i(s)^{\frac{2(2-n)}{n}}(-V_i'(s))ds.
$$
Hence, we establish that

$$\dfrac{4-n}{n}\int_{t}^{T_{u_i}}V_i(s)^{\frac{2(2-n)}{n}}\int_{\Gamma_s^{u_i}}\dfrac{dV_s}{|\nabla u_i|}ds \leq |\Omega|^{\frac{4-n}{n}}.$$
Finally, using this inequality, we deduce that

\begin{eqnarray*}
\|u_i\|_{L^{\infty}(\Omega)} & \leq & t + C(n)|\Omega|^{\frac{4-n}{2n}}\sqrt{\lambda_i}\left(\sum_k\frac{1}{\lambda_kt^2}\int_{\{u_k<t\}}|\nabla u_k|^4dx\right.\\
&& + \left.\sum_{k<l}\int_{\Omega\setminus\{u_k \geq t,u_l\geq t\}}f_{u_ku_l}(u)|\nabla u_k||\nabla u_l|dx\right)^{1/2}.
\end{eqnarray*}

\end{proof}

The next proposition consists of two estimates in a neighborhood of $\partial\Omega$ for positive strong solutions of (\ref{gradient-system}).

\begin{prop}\label{aux1}
Let $u$ be a positive strong solution of (\ref{gradient-system}). Denote $\Omega_{\varepsilon}=\{x\in\Omega : \delta(x)<\varepsilon\}$ for $\varepsilon > 0$. Assume that $\Omega$ is convex and $n\geq2$. Then, there exist constants $\varepsilon, D_1, D_2 > 0$ depending only on the domain $\Omega$ such that for any $i = 1, \ldots,m$,

\begin{itemize}
\item[(i)] $u_i(x)\geq D_1 \delta(x)\ \text{for all}\ x\in\Omega$;

\item[(ii)] $\|u_i\|_{L^{\infty}(\Omega_{\varepsilon})}\leq D_2 \|u_i\|_{L^1(\Omega)}$.
\end{itemize}

\end{prop}
The claim (i) is a direct consequence from Lemmas \ref{uniqueness-weak-solutions-lemma} and \ref{HL-estimate}. The tool used in the proof of the second assertion is the well-known moving planes method. We refer for example to Troy \cite{Troy1981} where the estimate (ii) is proved on convex domains for any $m \geq 1$ and $n\geq2$.

Propositions \ref{prop-boundary-estimative} and \ref{aux1} are key tools in the proof of the following result:

\begin{prop}\label{prop-boundedness}
Let $f\in C^{2}(\mathbb{R}^m)$ be a positive function satisfying $\Hessian f(t)>0$ for all $t \in\R^m_+$. Assume that the potential field $F = \nabla f$ satisfies (A), (B) and (C). Assume also that $\Omega$ is convex and $n=2, 3$. Let $u_{\Lambda}=(u_1,\ldots,u_m)$ be the stable minimal positive strong solution of (\ref{gradient-system}) for $\Lambda\in\mathcal{A}$. Then, there exists a constant $C_0 > 0$, depending on $\Omega, \varepsilon, D_1, D_2, \Lambda, \|\nabla f\|_{L^\infty\left(\overline{B}_r;\R^m\right)}$ and $\|\Hessian f\|_{L^\infty\left(\overline{B}_r;\R^{m^2}\right)}$, such that

\begin{gather}\tag{7.8}\label{final-estimate}
    \|u_i\|_{L^{\infty}(\Omega)}\leq C_0,
\end{gather}
where $\varepsilon$, $D_1$ and $D_2$ are given in Proposition \ref{aux1} and $r=D_2 \|u_{\Lambda}\|_{L^1(\Omega;\R^m)}$.
\end{prop}

\begin{proof}
In Proposition \ref{prop-boundary-estimative}, take

$$
t = D_1 \dfrac{\varepsilon}{2}.
$$
By the part (i) of Proposition \ref{aux1}, for $x \in \{u_i<t\}$, we have

$$
D_1 \delta(x) \leq u_i(x) < t = D_1 \dfrac{\varepsilon}{2},
$$
so that

$$
\delta(x)<\dfrac{\varepsilon}{2}.
$$
Thus,

\begin{gather}\tag{7.9} \label{subset}
\{u_i<t\}\subset\Omega_{\varepsilon/2}\ \ {\rm and}\ \ \Omega \setminus \{u_k \geq t, u_l \geq t\}\subset\Omega_{\varepsilon/2}
\end{gather}
for every $i, k, l=1,\ldots,m$. Since $n = 2, 3$, by Sobolev embedding, it suffices to establish the boundedness of $\|u_i\|_{W^{1,4}(\Omega_{\varepsilon/2})}$. Notice that $u_\Lambda$ satisfies

\begin{gather}\tag{7.10} \label{syst}
\left\{
\begin{array}{rlllr}
-{\Delta} u_\Lambda &=& \Lambda \nabla f(u_\Lambda) & {\rm in} & \Omega_\varepsilon, \\
u_\Lambda&=&0 & {\rm on} & \partial \Omega.
\end{array}\right.
\end{gather}
Moreover, $\partial\Omega\cup\Omega_{\varepsilon/2}$ is a precompact subset of $\partial\Omega\cup\Omega_{\varepsilon}$ and both are smooth. On the other hand, by the part (ii) of Proposition \ref{aux1}, we have

$$
\|f_{u_i}(u_1,\ldots,u_m)\|_{L^{\infty}(\Omega_\varepsilon)}\leq\|\nabla f\|_{L^\infty\left(\overline{B}_r;\R^m\right)}.
$$
So, $L^p$ Calderón-Zygmund estimate applied to each equation of \eqref{syst} yields

$$
\|u_i\|_{W^{1,4}(\Omega_{\varepsilon/2})}\leq C_1
$$
for some constant $C_1 > 0$ depending on $\Omega, \varepsilon, D_1, D_2, \Lambda$ and $\|\nabla f\|_{L^\infty\left(\overline{B}_r;\R^m\right)}$. Finally, Proposition \ref{prop-boundary-estimative}, \eqref{subset} and the above estimate give for any $i = 1, \ldots, m$,

\begin{equation*}
\|u_i\|_{L^{\infty}(\Omega)}\leq C_0.
\end{equation*}
\end{proof}
The boundedness of the extremal solution $u_\Lambda^*$ follows from Proposition \ref{prop-boundedness} as follows.

\begin{proof}[Proof of Theorem \ref{teo-extremal-strong}]
 Let $u_\Lambda^*$ be the extremal solution of (\ref{gradient-system}) associated to $\Lambda = (\lambda^*(\sigma), \lambda^*(\sigma) \sigma) \in \Lambda^*$. Take a sequence $\lambda_k$ converging to $\lambda^*(\sigma)$ such that $0 < \lambda_k < \lambda^*(\sigma)$. Let $u_{\Lambda_k}$ the stable minimal positive strong solution corresponding to $\Lambda_k = (\lambda_k, \lambda_k \sigma)$. Since $u_{\Lambda_k}$ converges pointwise almost everywhere in $\Omega$ and in $L^1(\Omega)$ to $u_\Lambda^*$, letting $k \rightarrow \infty$ in (\ref{final-estimate}), we deduce that $u_\Lambda^*\in L^{\infty}(\Omega)$. This concludes the proof.
\end{proof}

\n {\bf Acknowledgments:} The third author were partially supported by CNPq/Brazil (PQ 302670/2019-0, Universal 429870/2018-3) and Fapemig/Brazil (PPM-00561-18).


\end{document}